\documentclass{article}
\usepackage{amsmath,amssymb,amsthm,mathtools}
\usepackage{xspace}

\theoremstyle{plain}
\newtheorem{theorem}{Theorem}
\newtheorem{lemma}{Lemma}
\newtheorem{proposition}{Proposition}
\newtheorem{corollary}{Corollary}
\theoremstyle{definition}

\IfFileExists{latex_macros.sty}{
  \usepackage{latex_macros}
}{}
\providecommand{\real}{\mathbb{R}}
\providecommand{\Exp}{\mathbb{E}}
\providecommand{\IdMat}{I}

\providecommand{\inprod}[2]{\left\langle #1,#2\right\rangle}
\providecommand{\mycomment}[1]{}

\usepackage{fullpage}
\usepackage{xcolor}
\IfFileExists{fontawesome5.sty}{
  \usepackage{fontawesome5}

}{

}

\usepackage[
    colorlinks=true,
    linkcolor=magenta,
    citecolor=blue,
    urlcolor=blue
]{hyperref}

\usepackage[capitalize,noabbrev,nameinlink]{cleveref}
\crefname{assumption}{Assumption}{Assumptions}
\Crefname{assumption}{Assumption}{Assumptions}

\usepackage{enumitem}

\newcommand{\Xvar}{\ensuremath{X}}
\newcommand{\usedim}{d}

\newcommand{\defn}{\coloneqq}
\newcommand{\Normal}{N}

\newcommand{\Gvar}{G}
\newcommand{\Yvar}{Y}

\newcommand{\Bvar}{B}
\newcommand{\lam}{\lambda}

\newcommand{\Ylam}{\Yvar_\lam}

\newcommand{\Denoise}{m}
\newcommand{\Info}{\ShanInfo}

\newcommand{\Jinfo}{\mathsf{G}}
\newcommand{\Hinfo}{\mathsf{H}}
\newcommand{\Dinfo}{\mathsf{D}}
\newcommand{\Escore}{\mathsf{E}_{\scaleto{\operatorname{sc}}{3pt}}}

\newcommand{\EulKL}{\Gamma_{\scaleto{\operatorname{Eul}}{4pt}}}

\newcommand{\Cov}{\operatorname{cov}}

\providecommand{\Cmat}{\ensuremath{\mathbf{C}}}
\newcommand{\Amat}{\ensuremath{\mathbf{A}}}
\newcommand{\Var}{\operatorname{Var}}
\newcommand{\Tr}{\operatorname{tr}}

\newcommand{\OpNorm}[1]{\left\|#1\right\|_{\mathrm{op}}}
\newcommand{\KL}{\ensuremath{D_{\scaleto{\operatorname{KL}}{4pt}}}}

\newcommand{\Score}{\mathsf{s}}
\newcommand{\scorehat}{\widehat{\score}}
\newcommand{\score}{\Score}

\newcommand{\Hblk}{\mathsf H}
\newcommand{\ValFun}{\mathsf{V}}
\newcommand{\gfun}{\mathsf g}
\newcommand{\hfun}{\mathsf h}

\newcommand{\rhat}{\widehat r}

\newcommand{\Zvar}{Z}
\newcommand{\Qhat}{Q}
\newcommand{\Nscore}{N}
\newcommand{\NscoreTot}{\ensuremath{\Nscore_{\scaleto{\mathrm{tot}}{3pt}}}}
\newcommand{\NscoreInit}{\ensuremath{\Nscore_{\scaleto{\mathrm{init}}{3pt}}}}
\newcommand{\Ninit}{\ensuremath{\NscoreInit( \tfrac{\varepsilon}{2}; \Tfinal)}}
\newcommand{\Nwainwright}{\ensuremath{N_{\scaleto{\mbox{Wai}}{3pt}}}}
\newcommand{\NSLC}{N_{\scaleto{\mbox{SLC}}{4pt}}}

\newcommand{\Zhat}{\ensuremath{\widehat{Z}}}

\newcommand{\ShannonInfo}{\ensuremath{\operatorname{Info}}}
\newcommand{\ShanInfo}{\ShannonInfo}

\newcommand{\lamzero}{\ensuremath{\lam_0}}
\newcommand{\lamone}{\ensuremath{\lam_1}}

\newcommand{\Xtil}{\ensuremath{\widetilde{X}}}

\newcommand{\cov}{\Cov}

\providecommand{\Order}{\mathcal{O}}

\newcommand{\Blam}{\ensuremath{B_\lam}}

\newcommand{\Prob}{\ensuremath{\mathbb{P}}}
\newcommand{\Qprob}{\ensuremath{\mathbb{Q}}}

\newcommand{\Yhat}{\widehat{Y}}

\newcommand{\YhatS}[1]{\Yhat_{#1}}

\renewcommand{\var}{\ensuremath{\Var}}

\newcommand{\CovZ}{\boldsymbol{\Sigma}_Z}

\newcommand{\SIEuler}{SI-Euler\xspace}
\newcommand{\SIEulerRho}{SI-Euler$(\rho)$\xspace}

\newcommand{\Zset}{\ensuremath{\mathcal{Z}}}

\newcommand{\ShanDist}{\ensuremath{\mathcal{R}}}

\newcommand{\Gauss}{\ensuremath{\mathbb{G}}}

\newcommand{\Ent}{\ensuremath{\operatorname{Ent}}}
\newcommand{\polylog}{\operatorname{polylog}}

\newcommand{\Pcar}{C_{\scaleto{P}{4pt}}}

\newcommand{\Poincare}{Poincar\'{e}\xspace}

\newcommand{\Sprob}{\ensuremath{\mathbb{S}}}
\newcommand{\Partition}{\ensuremath{\mathcal{P}}}

\newcommand{\PartH}{\ensuremath{\PartComp(\Partition)}}
\newcommand{\PartComp}{\mathsf{C}_{\scaleto{\operatorname{\dgc}}{4pt}}}
\newcommand{\PartCompHat}{\widehat{\mathsf{C}}_{\scaleto{\operatorname{\dgc}}{4pt}}}

\newcommand{\Len}{\ensuremath{\operatorname{Len}}}

\newcommand{\Xhat}{\widehat{X}}
\newcommand{\DenTime}{\mu}

\newcommand{\iter}{j}
\newcommand{\MaxIter}{N}

\newcommand{\bind}{k}
\newcommand{\Btot}{K}

\newcommand{\block}{b}
\newcommand{\jind}{\iter}
\newcommand{\Ptrue}{P}

\newcommand{\Tinit}{\ensuremath{\delta}}
\newcommand{\Tfinal}{\ensuremath{T}}

\newcommand{\dyadic}{\ensuremath{v}}

\newcommand{\DyaTot}{\ensuremath{L}}

\newcommand{\Hupper}{\overline{\Hinfo}}
\newcommand{\Hhat}{\widehat{\Hinfo}}

\newcommand{\eigval}{\ensuremath{\lambda}}

\newcommand{\dgclong}{denoising growth complexity\xspace}
\newcommand{\dgc}{\textsf{DGC}\xspace}

\newcommand{\numobs}{\ensuremath{m}}
\newcommand{\Zup}[1]{\ensuremath{\Zvar^{(#1)}}}

\newcommand{\Qfunup}[1]{\Qfun^{(#1)}}

\newcommand{\HackErr}{\widehat{r}_\numobs(\eta)}
\newcommand{\NewHackErr}{\widehat{r}_{\bind, \numobs}(\eta/\Btot)}

\newcommand{\Sinfo}{\ensuremath{\mathsf{S}}}

\newcommand{\dmin}{\ensuremath{d_{\scaleto{\operatorname{eff}}{4pt}}}}
\newcommand{\Mani}{\ensuremath{\mathcal{M}}}
\newcommand{\mdim}{m}
\newcommand{\Wass}{\ensuremath{\mathcal{W}}}
\newcommand{\Ztil}{\ensuremath{\widetilde{Z}}}

\newcommand{\Vhat}{\ensuremath{\widehat{V}}}
\newcommand{\Nhat}{\ensuremath{\widehat{N}}}

\newcommand{\Nideal}{\ensuremath{N_{\scaleto{\operatorname{ideal}}{4pt}}}}

\newcommand{\qdens}{\mathsf{q}}

\newcommand{\DenTimeHat}{\widehat{\DenTime}}
\newcommand{\Hmod}{\ensuremath{\widetilde{\Hinfo}}}
\newcommand{\DinfoTil}{\ensuremath{\widetilde{\Dinfo}}}

\newcommand{\rhohat}{\ensuremath{\widehat{\rho}}}
\newcommand{\PartHfine}{\mathsf{P}_{\scaleto{\mathrm{\dgc}}{4pt}}}

\renewcommand{\Normal}{\ensuremath{\mathcal{N}}}
\newcommand{\DenError}{\ensuremath{\mathcal{E}}}

\newcommand{\Jtot}{J}

\newcommand{\fdens}{\ensuremath{\mathsf{f}}}
\renewcommand{\Len}{\ensuremath{L}}
\newcommand{\UniK}{\ensuremath{\Uni_\Btot}}
\newcommand{\Uni}{\ensuremath{\mathcal{U}}}
\newcommand{\Lip}{\ensuremath{\operatorname{Lip}}}

\makeatletter
\long\def\@makecaption#1#2{
        \vskip 0.8ex
        \setbox\@tempboxa\hbox{\small {\bf #1:} #2}
        \parindent 1.5em  
        \dimen0=\hsize
        \advance\dimen0 by -3em
        \ifdim \wd\@tempboxa >\dimen0
                \hbox to \hsize{
                        \parindent 0em
                        \hfil 
                        \parbox{\dimen0}{\def\baselinestretch{0.96}\small
                                {\bf #1.} #2
                                } 
                        \hfil}
        \else \hbox to \hsize{\hfil \box\@tempboxa \hfil}
        \fi
        }
\makeatother

\newenvironment{researchquestion}
  {\begin{center}\begin{minipage}{0.98\textwidth}}
  {\end{minipage}\end{center}}
  
\begin{document}

\begin{center}
  {\Large\bfseries
    Denoising growth complexity: \\    
    Data geometry and certified schedules for diffusion sampling\\
}

\vspace*{0.3in}

  \begin{tabular}{c}
    Martin J. Wainwright \\ \texttt{mjwain@mit.edu}
  \end{tabular}

  \vspace*{0.2in}    
  \begin{tabular}{c}
    Lab for Information and Decision Systems \\
    Statistics and Data Science Center \\       
    EECS and Mathematics, \\
    Massachusetts Institute of Technology 
  \end{tabular}
  
  \vspace*{0.25in}
  \today
  \vspace*{0.25in}

  \begin{abstract}
    Two central challenges in diffusion-based sampling are the
    theoretical one of understanding their remarkable effectiveness
    even in high-dimensional settings, and the practical one of
    designing algorithms with certified performance guarantees.  We
    show that these questions are intimately connected via the
    \emph{denoising growth complexity} ($\mathsf{DGC}$).  It is a
    geometric measure defined by a log-time weighted integral of the
    derivative of the denoising mean-squared error along the Gaussian
    heat flow.  We show how the $\mathsf{DGC}$ increments lead to a
    simple and explicit bound on the KL error of an Euler scheme
    applied to the stochastic innovations representation.  The bound
    is local along the path: each step is controlled by the
    corresponding \dgc increment and its relative stepsize.  This
    structure allows us to derive KL sampling guarantees for optimized
    stepsize schedules, both in a simpler single-block setting and in a
    more refined $K$-block setting. The \dgc function has a
    natural martingale structure, which we exploit to develop fully
    data-certified versions of these algorithms.  It also admits
    information-theoretic upper bounds in terms of covariance, rate
    distortion, metric entropy, and the Poincar\'e constant, thereby
    recovering and sharpening a range of existing diffusion-sampling
    guarantees, as well as giving new results.  In log heat-time, the
    fine partition limit is governed by an integral involving the square
    root of the $\mathsf{DGC}$ density, whereas a single-block schedule depends
    on its ordinary integral.  This comparison precisely characterizes
    when adaptation to data geometry yields substantial computational
    gains, including logarithmic-to-constant separations for simple
    Gaussian mixture models.
  \end{abstract}
\end{center}



\section{Introduction}

Diffusion-based methods have proven to be remarkably effective for
sampling from high-dimensional distributions, and they now underpin
powerful methods for generating images, audio, and other structured
data~\cite{RomEtAl22,CroEtAl23,YanEtAl25,CheEtAl24}.  Diffusion
samplers are built around some type of forward Gaussian noising
process~\cite{SohEtAl15,SonErmo19,HoEtAl20,SonEtAl21}. In continuous
time, the forward process is described by a stochastic differential
equation (SDE), and sampling can be performed by (approximately)
simulating the associated reverse-time SDE~\cite{HauPard86,And82}.  By
the Tweedie--Miyasawa formula (e.g.,~\cite{Rob56,Miy61,RapSim11}), the
score functions that govern the backward dynamics are determined by
the denoising functions, and so can be estimated via regression.
Given such (estimated) score functions, there are a variety of methods
for traversing the backward path, including stochastic samplers that
discretize the SDE
(e.g.,~\cite{HoEtAl20,SonEtAl21,LeeEtAl22,LiEtAl23,LeeEtAl23,CheEtAl23c,BenEtAl24});
and deterministic samplers that track the ordinary differential
equation (ODE) defining the probability flow
(e.g.,~\cite{SonEtAl21,BenEtAl23,AlbEtAl23,CheEtAl23e,CaiLi25,GatEtAl26}).

This paper is motivated by two questions associated with the analysis
and design of diffusion samplers.  A striking property of these methods
is their \emph{unreasonable effectiveness} relative to what generic
worst-case complexity bounds suggest.  This phenomenon appears to be
closely tied to structure in real-world distributions.  Although this
connection has been studied in various specific settings, the broader
question remains:
\begin{researchquestion} {\bf{Q1:}}
  Can the performance of diffusion sampling be explained and
  quantified, in some generality, by a measure tied to data geometry?
\end{researchquestion}
At the same time, a major practical challenge is to design sampling
procedures that have: (i) the best possible computational efficiency;
and (ii) outputs equipped with \emph{certified guarantees}.  Thus, we
are led to our second question:
\begin{researchquestion}
{\bf{Q2:}} Is it possible to exploit a diffusion-based measure of data
geometry to design and certify practical sampling schemes?
\end{researchquestion}
In this paper, we provide affirmative answers to both of these
questions, in particular via the notion of \emph{\dgclong}.


\subsection{Our contributions}

We now provide a high-level summary of the contributions of this
paper.  Focusing on a variant of a standard Euler discretization, our
main result (\Cref{ThmMaster}) gives an explicit upper bound on the KL
discretization error, which exposes the fundamental role of the
\dgclong.  Letting $t \mapsto \hfun(t)$ denote the mean-squared
denoising error along the heat path, the \dgc function $\Hinfo$ arises
as a log-time weighted integral of the derivative $\hfun'$; see
equation~\eqref{EqnDefnHinfo} for a precise definition.  As we show
in~\Cref{SecPropDGC}, the function $\Hinfo$ can be expressed in
information-theoretic terms, and its log-time derivative reveals the
geometry of the data distribution along the heat path.
See~\Cref{FigQplots,FigHier} for some representative plots.  Notably,
the proof of~\Cref{ThmMaster}, given in~\Cref{SecProofThmMaster},
requires fewer than three pages of elementary analysis.  Despite
this simplicity, we show that it recovers and sharpens a large number
of existing results on diffusion samplers.

The remainder of the paper develops various consequences
of~\Cref{ThmMaster} for different stepsize schedules.  We begin
in~\Cref{SecSingle} with the simplest single-block setting, where a
single geometric multiplier specifies stepsizes along the entire path.
\Cref{CorNewSingle}, which follows as an immediate consequence
of~\Cref{ThmMaster}, gives a precise upper bound on the KL sampling
error that depends on the \dgc function evaluated over the full path.
Notably, by exploiting the martingale structure of the \dgc function,
we also develop a fully data-certified version of this result, one
which uses trajectories of the heat path to estimate the \dgc
increments.  See~\Cref{SecDataSingle} and in
particular~\Cref{PropDataSingle} for details of this certification
procedure.  Next, in~\Cref{SecInfoUpper}, we develop various
model-specific consequences of our single-block theory, and show how
to recover and sharpen various results from past work, including sharp
dimension scaling; dimension-independent results in certain cases;
sharp dependence on support geometry and intrinsic dimensions; results
for manifold and mixture models in both exact and approximate
settings; and a novel result with logarithmic dependence on the
\Poincare constant. See~\Cref{SecRelated} for discussion of
connections to related work.

In~\Cref{SecTracking}, we harness the full power of~\Cref{ThmMaster}
by analyzing multi-block schedules, as well as the fine partition
limit thereof.  Our second main result (\Cref{ThmTrackMulti}) gives
explicit choices of geometric multipliers over a $\Btot$-partition
that achieve near-optimal iteration complexity for
$\varepsilon$-accurate sampling in KL divergence.  Notably, this
scheme remains amenable to data-certified guarantees;
in~\Cref{CorCertifiedMulti}, we specify a fully certified way of
choosing data-dependent multipliers that achieve the same guarantee up
to a factor of two, along with perturbations due to data
certification.  In~\Cref{SecTrackingBlock}, we show how to use dynamic
programming (DP) to construct optimal $\Btot$-block partitions; again,
the DP depends on \dgc increments that can be estimated via the
methods and guarantees of~\Cref{PropDataSingle}.  Finally,
in~\Cref{SecFine}, we analyze the fine partition limit, and prove a
partition-complexity sandwich (\Cref{LemFinePartSandwich}) that
quantifies the relation between a single-block scheme, its
$\Btot$-block refinement, and the limiting object. The latter object
depends on the log-time \dgc density $\qdens$, and it gives a crisp
characterization of when and how multi-block schemes lead to
    speed-ups.  In particular, the single-block iteration complexity
depends on an integral of $\qdens$, whereas the fine partition limit
depends on an integral involving $\sqrt{\qdens}$.  This
characterization, when coupled with the connection between the \dgc
density and data geometry in~\Cref{FigQplots}, shows the fundamental
link between questions {\bf{Q1}} and {\bf{Q2}} posed in the
introduction.


\subsection{Related work}
\label{SecRelated}

To put our results in context, let us now describe how they relate to
past work on diffusion sampling.  The literature now includes a wide
range of theoretical results, applying to both stochastic (SDE-based)
samplers
(e.g.,~\cite{LeeEtAl22,LiEtAl23,LeeEtAl23,CheEtAl23c,CheEtAl23a,BenEtAl24})
as well as (ODE or flow-based) deterministic ones
(e.g.,~\cite{SonEtAl21,AlbEtAl23,BenEtAl23,CheEtAl23e,CaiLi25,GatEtAl26}).
For either type of sampler, the dependence on the ambient dimension
$\usedim$ is now well-understood.  Initial results gave guarantees
with superlinear but polynomial dependence on
dimension~\cite{CheEtAl23a,LiEtAl23}, which were then refined to a
sharper linear dependence, up to logarithmic factors, in later
work~\cite{BenEtAl24,LiEtAl24}.  Our results show that the single
block scheme recovers linear scaling without any logarithmic overhead
via a covariance-based upper bound (\Cref{PropCovariance}) on the \dgc
complexity.  A notable exception to linear scaling is the recent
paper~\cite{GatEtAl26}, which analyzes a high-accuracy deterministic
sampler with only logarithmic dependence on dimension; under the same
assumptions, our results show that a stochastic Euler scheme, while
not giving the polylogarithmic $(1/\varepsilon)$-dependence of a
high-accuracy sampler, is in fact \emph{dimension-independent}.  See
the discussion following~\Cref{PropCovariance} for further details.

There is also a rich set of guarantees for diffusion samplers applied
to problems with some form of lower \emph{intrinsic dimension},
including manifold structures and other related conditions
(e.g.,~\cite{Bor22,Pid22,BorEtAl22,CheEtAl23d,AzaEtAl24}).  Various
researchers have shown that diffusion samplers can adapt to a lower
intrinsic dimension $k$, with the initial guarantees having
superlinear polynomial dependence on $k$~\cite{LiEtAl24b,AzaEtAl24},
and then refined to linear dependence, up to logarithmic factors, by
later work~\cite{PotEtAl24,HuaEtAl24a}.  Our results also capture this
linear scaling without any logarithmic overhead, and have natural
extensions to data distributions that lie close to a structure with
low intrinsic dimension.  Other work gives guarantees sampling
\emph{Gaussian mixture models} using diffusion methods
(e.g.,~\cite{ShaEtAl23,GatEtAl24,CheKonSha25,LiCaiWei26}).  Our
results capture the dimension-independent scaling given in the
paper~\cite{LiCaiWei26}, as well as logarithmic dependence on the
number of mixture components.  Within the framework of this paper, all
of these claims are captured via a general result
(\Cref{PropGoodRateDistor}) that connects the \dgc complexity to
low-distortion quantizers, and thereby to the Shannon rate-distortion
function of the data source.

Our work connects the typical heat-time analysis of diffusion samplers
with its equivalent parameterization in terms of stochastic
localization~\cite{Eld13,Mon23}.  In particular, the \dgc complexity
arises naturally from the Euler--Maruyama discretization of a
stochastic innovations SDE, specified as
equation~\eqref{EqnInnovationsSDE}.  This particular SDE, in turn,
arises from an application of classical nonlinear filtering
theory~\cite{KalStr68,FujKalKun72} to the stochastic localization
path.  While we focus only on the Euler discretization, a growing body
of work analyzes accelerated discretizations of reverse diffusion SDEs
including higher-order approximation, randomized-midpoint, and
stochastic Runge--Kutta
methods~\cite{LiEtAl24Accel,LiCai24,LiJia25,LiEtAl25Higher}.  Under
differing regularity assumptions and error metrics, these methods
improve the polynomial dependence on the target accuracy beyond the
$1/\varepsilon$ scaling of Euler-based guarantees.  A very recent line
of work~\cite{HuaEtAl24b,CheEtAl26,GatEtAl26,Wai26} has focused on
high-accuracy guarantees for diffusion samplers, meaning methods that
guarantee $\varepsilon$-accuracy with iteration complexity scaling
polylogarithmically in $1/\varepsilon$.  An interesting open question
is whether high-accuracy solvers can be adapted to the stochastic
innovations SDE that underlies our approach, thereby yielding refined
procedures that adapt to the data geometry and can be certified.

Portions of our analysis make use of the I--MMSE integral
identity~\cite{GuoEtAl05,GuoEtAl11} that relates the denoising
mean-squared error (MSE) function $\hfun$ to the mutual
information. In this way, it has connections to the
paper~\cite{ReePfi25}, which analyzes diffusions using
information-theoretic methods and the I--MMSE identity.  The \dgc
complexity identified in our work relies instead on a weighted
integral of the \emph{derivative} $\hfun'$ of the denoising MSE, which
arises from a locally multiplicative upper bound on the KL
discretization error (see~\Cref{LemGeoOneStep}).  Using the I-MMSE
equivalence, the \dgc function can also be related to mutual
information increments (see~\Cref{SecHinfoInfoRep}).  Hence, via this
connection, our guarantee for the special single-block scheme
(\Cref{CorNewSingle}) is connected to the guarantees in the
paper~\cite{ReePfi25}.  It should be noted that the fully general form
of~\Cref{ThmMaster} decomposes the full path into local increments of
the \dgc function; this additivity property leads to schemes that
outperform any single-block method.

Our multi-block results in~\Cref{SecTracking} provide a principled
basis for designing practical and fully certified stepsize schedules
that adapt to data geometry.  Some practical stepsize selection
schemes include fixed parametric rules, such as power-law grids in the
noise standard deviation~\cite{KarEtAl22}, or grids that are uniform
in solver-adapted coordinates such as log-SNR~\cite{LuEtAl22}.  An
evolving line of work~\cite{WatEtAl21,SabFidKre24,TonEtAl25} has
introduced solver- and model-specific optimization criteria, including
those based on ELBO bounds or path space KL divergence, which lead to
numerical procedures. In contrast, our schedules are derived
analytically from the \dgc profile itself, and come with direct
guarantees for the resulting KL sampling error of the output.
Moreover, we give data-dependent estimators for stepsize choices with
rigorous guarantees (see~\Cref{PropDataSingle}, as well
as~\Cref{CorCertifiedMulti}).


\section{General guarantee via \dgclong}
\label{SecGeneral}

In this paper, we analyze diffusion samplers that are based on the
\emph{Gaussian heat flow} $\{X_t \}_{t \geq 0}$, also referred to as
the forward heat path.  It is generated via $X_t = \Zvar + B_t$, where
$\Zvar$ is the target variable from which we wish to sample, and
$\{B_t\}_{t \geq 0}$ is a Brownian motion.  This choice is also known
as the variance-exploding parameterization, since the variance of
$X_t$ grows in time.  Sampling algorithms involve traversing the
reverse path, using estimates of the score function $\score_t(x_t)
\defn \nabla \log p_{t}(x_t)$ along the heat path.  By the
Tweedie--Miyasawa formula~\cite{Rob56,Miy61,RapSim11,Efr11}, this
score function is connected to the optimal denoiser; see
equation~\eqref{EqnTweedie} below.

A typical algorithm is based on traversing the heat path, moving from
a larger temperature $\Tfinal$ backwards to an initial $\Tinit \in (0,
\Tfinal)$.  In this paper, we analyze the following simple Euler-type
scheme.  Fix an arbitrary grid of distinct points on the interval
$[\Tinit, \Tfinal]$, moving backwards from $\Tfinal$ to $\Tinit$:

\begin{align}
\label{EqnGrid}  
  \Tinit = t_{\MaxIter} < t_{\MaxIter - 1} < \ldots < t_1 < t_0 =
  \Tfinal.
\end{align}
Beginning with some initial point $\Xhat_0 \sim \Qprob_\Tfinal$, we
analyze the following discrete-time updates:
\mygraybox{ {\bf{\SIEuler-scheme:}} For $\iter = 0, \ldots, \MaxIter -
  1$:
\begin{align}
\label{EqnHeatAlgorithm}
  \Xhat_{\iter +1} & = \Xhat_\iter + (t_\iter - t_{\iter+1})
  \score_{t_\jind}(\Xhat_{\jind}) + \sqrt{(t_\iter - t_{\iter+1})
    \frac{t_{\iter+1}}{t_\iter}} W_\iter \qquad \mbox{where $W_\iter \sim
    \Normal(0, \IdMat)$.}
\end{align}
}
To clarify the terminology \SIEuler, these updates are closely
related, but subtly distinct from the standard Euler--Maruyama
discretization of the reverse heat SDE.  The difference is the
presence of the factor $t_{\iter+1}/t_\iter$ multiplying the additive
Gaussian noise $W_\iter$.  As our analysis makes clear, the
algorithm~\eqref{EqnHeatAlgorithm} corresponds to the Euler--Maruyama
discretization of a stochastic innovations (SI) SDE from nonlinear
filtering (see equation~\eqref{EqnInnovationsSDE}), and the additional
factor arises from the conversion between the standard heat path and
the stochastic innovations path.  While we have described the
algorithm that uses exact score functions, our analysis also extends
to approximate score functions $\{\scorehat_t \}_{t \geq 0}$; see
equation~\eqref{EqnScoreError} below.


\subsection{Main guarantee}
\label{SecMain}

We now present the main guarantee of this paper, from which a number
of interesting consequences follow.  First, recall that by the
Tweedie--Miyasawa formula~\cite{Rob56,Miy61,RapSim11,Efr11}, the score
function $\score_t$ is connected to estimation of $\Zvar$ based on
$X_t$ via the relation
\begin{align}
\label{EqnTweedie}  
\score_t(x) & = \frac{\DenTime_{t}(x) - x}{t} \qquad \mbox{where
  $\DenTime_t(x) \defn \Exs[\Zvar \mid X_t = x]$.}
\end{align}
Here the denoising operator $x \mapsto \DenTime_t(x)$ defines the
minimum mean-squared error estimate of $\Zvar$ given the observation
$\Xvar_t = x$.

Our theory involves the associated mean-squared error (MSE) function
\begin{subequations}
\begin{align}
  \label{EqnHeatMSE}
  \hfun(t) & \defn \Exs \| \Zvar - \DenTime_t(X_t) \|_2^2 \qquad
  \mbox{where $X_t = Z + B_t$,}
\end{align}
and we assume that the derivative $\hfun'$ exists and is integrable.
For any pair $0 < a < b < \infty$, we define the \emph{\dgclong}, or
\dgc for short, via
\begin{align}
\label{EqnDefnHinfo}  
\Hinfo(a,b) & \defn \frac{1}{2} \int_a^b \frac{\hfun'(t)}{t} dt.
\end{align}
\end{subequations}
Letting $\Xhat_N \sim \Qprob_\Tinit$ denote the distribution of the
algorithm's~\eqref{EqnHeatAlgorithm} final iterate, our main result is
a very simple upper bound on the KL divergence $\KL(\Prob_\Tinit \|
\Qprob_\Tinit)$ in terms of the \dgc function along the path:
\mygraybox{
\begin{theorem}
\label{ThmMaster}
Suppose that $\Zvar$ has finite second moments, and the MSE function
$\hfun$ is differentiable. Then given an initial value $\Xhat_0 \sim
\Qprob_\Tfinal$, running the \SIEuler scheme using any
$\{t_\iter\}_{\iter=0}^{\MaxIter}$ grid over $[\Tinit, \Tfinal]$
yields a terminal output $\Xhat_\MaxIter \sim \Qprob_\Tinit$ such that
\begin{align}
\label{EqnMaster}  
\KL(\Prob_\Tinit \| \Qprob_\Tinit) & \leq \sum_{\iter =
  0}^{\MaxIter-1} \big( \tfrac{t_\iter}{t_{\iter + 1}} - 1 \big)
\Hinfo(t_{\iter+1}, t_\iter) + \KL(\Prob_\Tfinal \| \Qprob_\Tfinal).
  \end{align}
  \end{theorem}
}

\noindent See~\Cref{SecProofThmMaster} for the proof.  Notably, it is
less than three pages long, and involves only elementary analysis.
Here we comment on a few features of this bound.

\paragraph{Incorporating score error:}
As stated, \Cref{ThmMaster} applies to the idealized algorithm that
uses exact score functions $\score_t$ to perform the
updates~\eqref{EqnHeatAlgorithm}.  In practice, the algorithm might be
implemented with estimates $\scorehat_t$ of the score function at each
time $t$.  Our result remains robust to such errors in the standard
way: in particular, the upper bound~\eqref{EqnMaster} remains valid
after the addition of the score error term
\begin{align}
\label{EqnScoreError}
  \Escore(t_0^{\MaxIter-1}) & \defn \frac{1}{2} \sum_{\iter =
    0}^{\MaxIter - 1} \big \{ \tfrac{t_\iter}{t_{\iter+1}} - 1 \big \}
  \; t_\iter \; \Exs \big[ \| \scorehat_{t_\iter}(X_{t_\iter}) -
    \score_{t_\iter}(X_{t_\iter}) \|_2^2 \big] \qquad \mbox{where
    $X_{t_\iter} \sim \Prob_{t_\iter}$.}
  \end{align}
We establish this fact as part of the proof
in~\Cref{SecProofThmMaster}.

\paragraph{Initialization cost:}  The initialization error
$\KL(\Prob_\Tfinal \| \Qprob_\Tfinal)$ is straightforward to handle,
with several choices available.  The simplest is to initialize the
algorithm with a sample $\Xtil_\Tfinal \sim \Gauss_\Tfinal$, where
$\Gauss_\Tfinal$ denotes the $N(0, \Tfinal \IdMat)$ distribution, so
that no further score evaluations are required.  Letting $\CovZ \defn
\cov(Z)$ denote the covariance matrix of $Z$, an elementary
calculation shows that, under the centering assumption $\Exs[\Zvar] =
0$, we have $\KL(\Prob_\Tfinal \| \Gauss_\Tfinal) \leq \frac{\Tr
  \CovZ}{2 \Tfinal}$, so that setting $\Tfinal = \frac{\Tr
  \CovZ}{\varepsilon}$ ensures that the initialization has KL-error at
most $\varepsilon/2$.

This naive initialization leads to a choice of $\Tfinal$ that depends
on $\varepsilon$.  This can be avoided by making use of more
sophisticated results that determine the smallest $\Tfinal$ such that
$\Prob_\Tfinal$ is strongly log-concave and well-conditioned, in which case
we can obtain an $\varepsilon/2$ accurate initialization using
$\Order(1/\varepsilon)$ rounds of the unadjusted Langevin algorithm
(ULA) at the terminal stage.


\subsection{Properties of \dgclong}
\label{SecPropDGC}

The \dgc function~\eqref{EqnDefnHinfo}, which plays a central role
in~\Cref{ThmMaster}, has a number of interesting properties that are
conceptually valuable, and turn out to be extremely useful for
algorithmic tuning.

\subsubsection{Local sharpness and additivity}
\label{SecLocalSharpness}
As shown in~\Cref{SecProofThmMaster}, the proof of~\Cref{ThmMaster}
relies upon an upper bound on the error in the Euler
discretization~\eqref{EqnHeatAlgorithm}.  In particular, for $0 < \eta
< s$, letting $\EulKL(s - \eta, s)$ denote the KL error incurred by
the discretization over the interval $[s - \eta, s]$, we prove the
upper bound
\begin{subequations}
  \begin{align}
\label{EqnNewEuler}    
\EulKL(s - \eta, s) & \leq \psi \Big(\frac{\eta}{s} \Big) \, \; \Hinfo(s -
\eta, s) \qquad \mbox{where $\psi(t) = \dfrac{t}{1-t}$.}
\end{align}
  See~\Cref{LemGeoOneStep} for the details of this claim.  To connect
  with~\Cref{ThmMaster}, observe that for the choices $s = t_{j}$ and
  $\eta = t_j - t_{j+1}$, we have $\eta/s = 1 - (t_{j+1}/t_j)$, and
  $\psi(\eta/s) = \frac{t_j}{t_{j+1}} - 1$, consistent with the terms
  in equation~\eqref{EqnMaster}.

Given the upper bound~\eqref{EqnNewEuler}, it is natural to wonder
whether or not $\Hinfo$ is fundamental, or just a convenient device
for analysis. It turns out that, under some mild regularity
conditions, the $\Hinfo$-based upper bound~\eqref{EqnNewEuler} is
sharp up to a constant factor as the interval shrinks.  In particular,
when $\hfun'$ is continuous at $s$ with $\hfun'(s) > 0$, a Taylor
argument can be used to show that
\begin{align*}
\frac{\EulKL(s - \eta, s)}{\psi(\eta/s) \; \Hinfo(s - \eta, s)}
  & = \frac{1}{2} + o(1) \quad \mbox{as $\eta \rightarrow 0$.}
\end{align*}
Thus, the \dgc function is fundamental, in the sense that it recovers
the KL error up to a factor of two in the short interval limit, under
mild regularity requirements.

For future reference, we note that for any triple $0 < a < b < c$, it
is an immediate consequence of the definition~\eqref{EqnDefnHinfo} that
the \dgc function satisfies the additivity property
\begin{align}
  \label{EqnHinfoAdditive}
  \Hinfo(a, c) & = \Hinfo(a, b) + \Hinfo(b, c).
\end{align}
\end{subequations}
This simple property is very useful, since it allows us to decompose
arbitrary intervals into sums over shorter intervals.


\subsubsection{Information representation}
\label{SecHinfoInfoRep}

The \dgc function admits an equivalent representation in terms of the
information lost along the heat path.  While less practically useful
than the integral representation~\eqref{EqnDefnHinfo}, this
information-theoretic representation provides conceptual insight, and
can be used to derive various analytic upper bounds. Letting $\Info(U;
V)$ denote the mutual information between random variables $U$ and
$V$, it can be shown that the \dgc function~\eqref{EqnDefnHinfo} has
the equivalent representation
\begin{align}
  \label{EqnHinfoInfoRep}
  \Hinfo(a, b) & = \Info(\Zvar; X_a) - \Info(\Zvar; X_b) + \frac{1}{2}
  \left\{ \frac{\hfun(b)}{b} - \frac{\hfun(a)}{a} \right\}.
\end{align}
Thus, we see that $\Hinfo(a, b)$ measures the decrease in information
about the latent variable $\Zvar$ between noise levels $a$ and $b$,
together with an explicit correction determined by the endpoint
denoising errors.  As shown in~\Cref{SecProofEqnHinfoRep}, the
equivalence~\eqref{EqnHinfoInfoRep} follows by integration by parts,
combined with the I--MMSE identity~\cite{GuoEtAl05}.


\subsubsection{The log-time \dgc density}
\label{SecLogTimeDGC}

Let us now introduce a fundamental quantity that plays a role in the
sequel.  In particular, we define
\begin{align}
\label{EqnDefnOrigQdens}  
\qdens(r) & \defn 2 \frac{d}{dr} \Hinfo(\Tinit, \Tinit e^r),
\end{align}
corresponding to the \emph{log-time \dgc density}.  In terms of
$\hfun'$, it can be verified that we have the equivalence $\qdens(r) =
\hfun'(\Tinit e^r)$.  In~\Cref{SecFine}, we discuss the key role that
this density, along with its square root $\sqrt{\qdens(r)}$, plays in
understanding the accuracy of the general schemes
from~\Cref{ThmMaster}.

In addition to these analytical properties, the log-time \dgc density
provides insight into the geometric structure of the latent variable
$\Zvar$ as we illustrate via some simple examples here.
\begin{figure}[h]
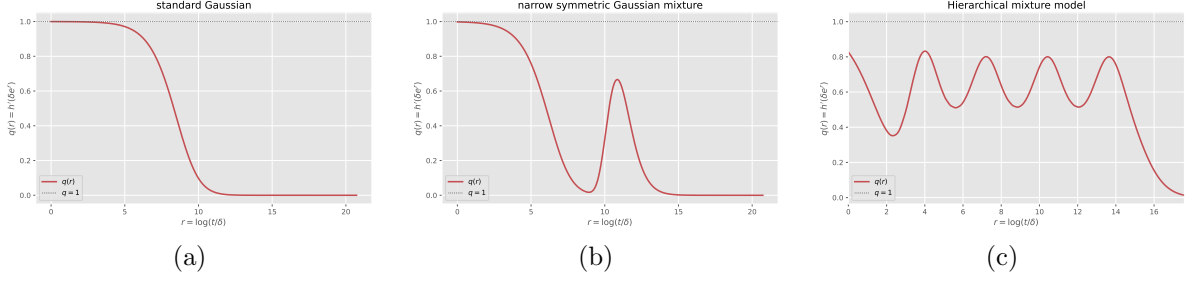

  \begin{center}
    \begin{tabular}{ccc}
      \widgraph{0.3\textwidth}{\figdir/fig_standard_gaussian} &
      \widgraph{0.3\textwidth}{\figdir/fig_two_gaussian} &
      \widgraph{0.3\textwidth}{\figdir/fig_hier_gmm}
      \\
      (a) & (b) & (c)
    \end{tabular}
    \caption{Plots of the log-time \dgc
      density~\eqref{EqnDefnOrigQdens} for three different choices of
      the latent variable $\Zvar$.  (a) A standard Gaussian
      $\Normal(0,1)$.  (b) A Gaussian mixture model with $2$
      components.  (c) A hierarchical Gaussian mixture model with the
      nested structure shown in~\Cref{FigHier}.}
        \label{FigQplots}
  \end{center}
\end{figure}
\Cref{FigQplots} shows plots of the function $\qdens$ for three
representative instances of $\Zvar$, all univariate variables in this
simple illustration.  Panel (a) corresponds to the standard normal
$\Normal(0,1)$; in this case, we see that the log-time \dgc density is a
very simple and monotonic function.  In panel (b), we plot $\qdens$
for a $2$-component Gaussian mixture model, with components chosen
with relatively small variances.  This choice leads to a $\qdens$
function with a sharp peak in the middle, corresponding to the
log-heat time at which the two mixture components are resolved.

\begin{figure}[h]
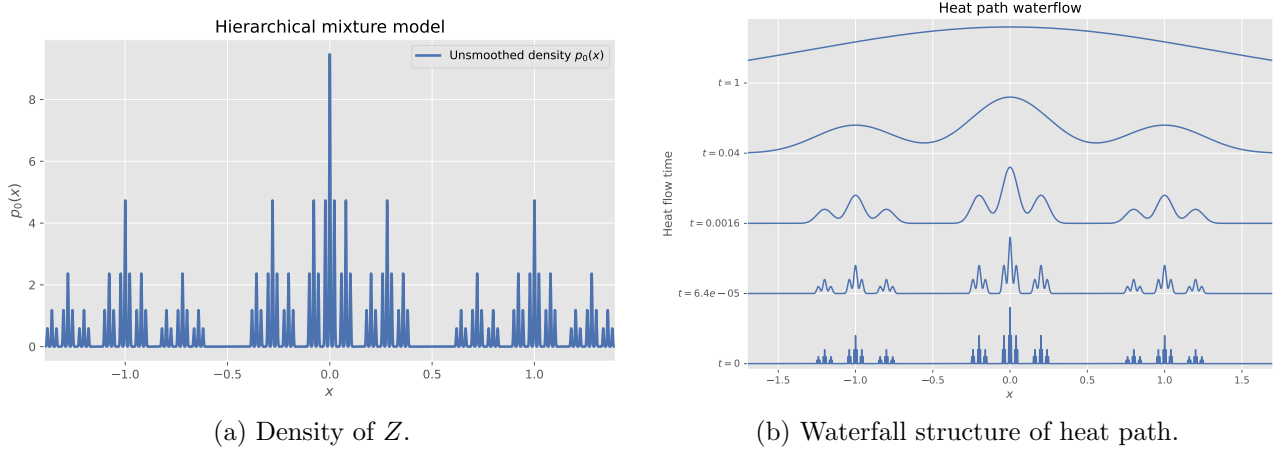

  \begin{center}
    \begin{tabular}{cc}
      \widgraph{0.5\textwidth}{\figdir/fig_full_density_hier} &
      \widgraph{0.5\textwidth}{\figdir/fig_heat_waterfall_hier} \\ (a)
      Density of $\Zvar$. & (b) Waterfall structure of heat path.
    \end{tabular}
    \caption{Structure of the hierarchical mixture model.  (a) The
      density is separated into mixture components that are resolved
      at a sequence of scales. (b) Waterfall structure of the heat
      path, showing the resolution of mixture components of $\Zvar$ at
      four distinct heat times.}
        \label{FigHier}
  \end{center}
\end{figure}

Finally, panel (c) in~\Cref{FigQplots} shows a more interesting
example of $\qdens$, with a multiscale geometry carefully designed to
generate a \dgc density with $L = 4$ well-separated peaks.  The
structure of the underlying latent variable $\Zvar$ is shown
in~\Cref{FigHier}.  As can be seen from the density in panel (a), it
is a hierarchical Gaussian mixture model, with a total of $3^4$
mixture components, organized into a hierarchy of $L = 4$ levels.
Panel (b) shows the waterfall structure of the heat path $X_t = \Zvar
+ B_t$, with different resolutions of structure revealed at different
heat times $t$.  In particular, there are four distinct times at which
the heat path resolves a particular hierarchy of the mixture model.
By the latest heat time $t = 0.04$, the smoothed density has
become unimodal, whereas the density at the initial heat time has a
highly complex structure.  It is this hierarchical structure that leads to $L
= 4$ distinct bumps in the $\qdens$-profile in~\Cref{FigQplots}(c),
and a total of $C^L = 3^4$ components in the original GMM
in~\Cref{FigHier}(a).

\subsubsection{Denoising increments and data-based estimation}
\label{SecMartingale}

Finally, let us elucidate an important connection between the \dgc
function and the structure of the heat path, one that allows us to
develop data-dependent procedures with certified guarantees.  Given
the forward heat path $X_t \defn Z + B_t$, where $\{B_t \}_{t \geq 0}$
is a Brownian motion, we define the \emph{denoising increment}
\begin{align}
\label{EqnDefnDinfo}
\Dinfo(s, t) & \defn \Exs \| \DenTime_s(X_s) - \DenTime_t(X_t) \|_2^2,
\qquad \mbox{for any pair $0 < s < t$,}
\end{align}
where we recall from equation~\eqref{EqnTweedie} the definition
$\DenTime_t(x) \defn \Exs[ \Zvar \mid X_t = x]$ of the denoiser.  By
construction, the denoising increment $\Dinfo(s,t)$ lends itself
naturally to a Monte Carlo estimate, assuming that we have samples of
the latent variable $Z$ available for constructing the forward heat
path. The following result relates the denoising increment to the MSE
function $\hfun(t) = \Exs \|Z - \DenTime_t(X_t)\|_2^2$, and shows how
it can be used to sandwich the \dgclong $\Hinfo$.

\mygraybox{
\begin{lemma}
\label{LemDinfoSandwich}
For any pair $0 < s < t$, we have the equivalence
\begin{subequations}
  \begin{align}
\label{EqnDinfoExplicit}    
 \Dinfo(s, t) = \hfun(t) - \hfun(s) \qquad \mbox{where $\hfun(t) =
   \Exs \|\Zvar - \DenTime_t(X_t)\|_2^2$.}
\end{align}
as well as the sandwich relation
\begin{align}
\label{EqnDinfoSandwich}  
\frac{\Dinfo(s, t)}{t} \; \leq \; 2 \Hinfo(s, t) \leq \frac{\Dinfo(s,
  t)}{s}.
\end{align}
\end{subequations}
\end{lemma}
}

\paragraph{Utility of denoising increments:}  
The sandwich property~\eqref{EqnDinfoSandwich} is particularly useful
when applied to an interval of a given multiplicative length---i.e.,
for pairs $s$ and $t = c s$ for some constant $c > 1$.  In this
special case, the sandwich property yields the \emph{constant-factor
approximation}
\begin{subequations}
\begin{align}
\label{EqnConstantFactor}  
\frac{1}{c} \; \frac{\Dinfo(s, t)}{s} \; \leq \; 2 \Hinfo(s, t) \;
\leq \; \frac{\Dinfo(s, t)}{s}.
\end{align}
Moreover, by leveraging the additivity
property~\eqref{EqnHinfoAdditive}, we can obtain constant-factor
approximations over blocks of arbitrary length.  We make use of both
Monte Carlo approximations and these constant-factor approximations in
our analysis of data-based methods for selecting a single stepsize
in~\Cref{SecDataSingle}, and for multiple stepsizes
in~\Cref{SecDataMultiple}.

\paragraph{Martingale structure of the heat path:}
The proof of~\Cref{LemDinfoSandwich} makes use of properties of the
forward heat path $X_t = Z + B_t$. In particular, for any $t > s$, we
have the Markov relation $Z \rightarrow X_s \rightarrow X_t$, and
hence we have $\Exs[\Zvar \mid X_s, X_t] = \Exs[\Zvar \mid X_s] =
\DenTime_s(X_s)$.  As a result, we have the equivalence
\begin{align}
\label{EqnReverseMartingale}  
\Exs[ \DenTime_s(X_s) \mid X_t] & = \Exs \left[ \Exs[\Zvar \mid X_s,
    X_t] \mid X_t \right] \; = \; \Exs[\Zvar \mid X_t] \; = \;
\DenTime_t(X_t),
\end{align}
\end{subequations}
which can be understood as a (reverse) martingale property for the
denoising functions $\{\DenTime_t(X_t) \}_{t \geq 0}$. As shown
in~\Cref{SecProofLemDinfoSandwich}, the
equivalence~\eqref{EqnDinfoExplicit} can be established by using this
martingale property in conjunction with the Pythagorean theorem.

As for the sandwich relation~\eqref{EqnDinfoSandwich}, it follows from
the definition~\eqref{EqnDefnHinfo} of the \dgc function, combined
with the fact that $\hfun'(u) \geq 0$, corresponding to the fact that
the MSE is non-decreasing along the forward heat path.  More
precisely, we have
\begin{align*}
\Hinfo(s,t) & \stackrel{(i)}{=} \frac{1}{2} \int_s^t
\frac{\hfun'(u)}{u} du \; \stackrel{(ii)}{\leq} \frac{1}{2 s} \int_s^t
\hfun'(u) du \; = \; \frac{1}{2} \frac{\hfun(t) - \hfun(s)}{s},
\end{align*}
where step (i) follows from the definition~\eqref{EqnDefnHinfo} of
$\Hinfo$; step (ii) follows since $h'(t) \geq 0$ and $1/u \leq 1/s$
for $u \in [s,t]$, and the final step follows by computing the
integral.  This argument establishes the upper bound in the
sandwich relation~\eqref{EqnDinfoSandwich}, and the lower bound
follows by an entirely analogous argument.


\section{Single-block guarantees}
\label{SecSingle}

In this section, we explore a simple but easily interpretable
consequence of the general \dgc guarantee in~\Cref{ThmMaster}.  It
involves a geometrically decaying choice of the stepsizes based on a
\emph{single parameter} $\rho$, and it leads to bounds that depend on
the \dgc complexity $\Hinfo(\delta, T)$ of the full interval $[\delta,
  T]$.  In~\Cref{SecSingleGuarantee}, we derive this upper bound as a
direct corollary of~\Cref{ThmMaster}.  In~\Cref{SecDataSingle}, we
exploit the properties of the \dgc function to derive data-adaptive
parameter choices that achieve our upper bound up to constant factors.
Finally, \Cref{SecInfoUpper} is devoted to information-theoretic upper
bounds on $\Hinfo(\delta, T)$; exploration of consequences for
specific model classes; as well as comparison to related results.

\subsection{Guarantees for a geometric schedule}
\label{SecSingleGuarantee}

Suppose that for some $\rho > 0$, we initialize with $t_0 = \Tfinal$,
and choose grid points according to the decreasing geometric schedule
\begin{align}
  \label{EqnDecreasingGeometric}
  t_{\iter + 1} = \max \Big \{ \frac{t_\iter}{1 + \rho}, \Tinit \Big
  \} \qquad \mbox{terminating at the smallest $\MaxIter$ such that
    $t_\MaxIter = \Tinit$.}
\end{align}
We refer to the updates~\eqref{EqnHeatAlgorithm} using this geometric
schedule as the \emph{\SIEulerRho procedure}, and the following result
provides a direct analysis of its KL error.

\mygraybox{
  \begin{corollary}
\label{CorNewSingle}
For any given iteration budget $\Nscore \geq \log(\Tfinal/\Tinit)$,
suppose that we set $\rho = \exp \{
\frac{\log(\Tfinal/\Tinit)}{\Nscore} \} - 1$.  Then the \SIEulerRho
procedure yields terminal output $X_N \sim \Qprob_\Tinit$ with
\begin{align}
\label{EqnNewSingle}
  \KL(\Prob_\Tinit \| \Qprob_\Tinit) & \leq \frac{2 \,\Hinfo(\Tinit,
    \Tfinal) \log(\Tfinal/\Tinit)}{\Nscore} + \KL(\Prob_\Tfinal \|
  \Qprob_\Tfinal).
\end{align}
  \end{corollary}
}
\noindent The proof is an easy consequence of~\Cref{ThmMaster}, so we
give it here.
\begin{proof}
Our choice of $\rho$ and the geometric
schedule~\eqref{EqnDecreasingGeometric} ensures that we traverse the
interval $[\Tinit, \Tfinal]$ in exactly $\Nscore$ steps.  Moreover,
the geometric schedule ensures that $\tfrac{t_\iter}{t_{\iter + 1}} -
1 \leq \rho$, so that the bound~\eqref{EqnMaster} then implies that
\begin{align}
  \label{EqnMasterSingle}
  \KL(\Prob_\Tinit \| \Qprob_\Tinit) & \leq \rho \sum_{\iter =
    0}^{\MaxIter - 1} \Hinfo(t_{\iter + 1}, t_\iter) +
  \KL(\Prob_\Tfinal \| \Qprob_\Tfinal) \; \stackrel{(i)}{=} \; \rho
  \Hinfo(\Tinit, \Tfinal) + \KL(\Prob_\Tfinal \| \Qprob_\Tfinal),
\end{align}
where equality (i) follows from the fact that $\sum_{\iter =
  0}^{\MaxIter - 1} \Hinfo(t_{\iter + 1}, t_\iter) = \Hinfo(\Tinit,
\Tfinal)$, using the additivity property~\eqref{EqnHinfoAdditive}.
Finally, using the upper bound $e^t - 1 \leq 2 t$, valid for $t \in
[0,1]$, we observe that $\rho \leq 2
\frac{\log(\Tfinal/\Tinit)}{\Nscore}$ when $\Nscore \geq
\log(\Tfinal/\Tinit)$.
\end{proof}

\paragraph{Iteration complexity:}  For future reference,
it is also useful to restate the bound~\eqref{EqnNewSingle} in terms
of the iteration complexity.  Let $\Ninit$ denote the iteration
complexity of a procedure for drawing $\Xhat_0 \sim \Qprob_\Tfinal$
such that $\KL(\Prob_\Tfinal \| \Qprob_\Tfinal) \leq \varepsilon/2$.
We then apply the procedure in~\Cref{CorNewSingle} with the iteration
budget
\begin{subequations}
  \begin{align}
\label{EqnSingleBudget}
N(\varepsilon/2) & \defn \left \lceil \max \left \{ \log(T/\delta),
\frac{4 \Hinfo(\Tinit, \Tfinal) \log(\Tfinal/\Tinit)}{\varepsilon}
\right \} \right \rceil \; \leq \; 1 + \log(\Tfinal/\Tinit) +
\frac{4}{\varepsilon} \; \Hinfo \big(\Tinit, \Tfinal\big) \: \log
\big(\tfrac{\Tfinal}{\Tinit} \big),
\end{align}
thereby ensuring that the first term in the bound~\eqref{EqnNewSingle}
is at most $\varepsilon/2$.  Combining the ingredients, we can sample
to $\varepsilon$-accuracy in KL divergence using at most
\begin{align}
\label{EqnSingleBlockStrong}  
\NscoreTot(\varepsilon) & \leq 1 + \log(\Tfinal/\Tinit) +
\frac{4}{\varepsilon} \; \Hinfo \big(\Tinit, \Tfinal\big) \: \log
\big(\tfrac{\Tfinal}{\Tinit} \big) + \Ninit
\end{align}
\end{subequations}
iterations.


\subsection{Data-dependent method for certified guarantees}
\label{SecDataSingle}

Given a fixed iteration budget $N$, we can achieve the
bound~\eqref{EqnNewSingle} from~\Cref{CorNewSingle} with a fully
implementable scheme, since all elements of the triple $(N, \Tinit,
\Tfinal)$ that determine the geometric multiplier $\rho$ are known.
However, the more challenging and user-relevant problem is to give a
\emph{certified guarantee} on the accuracy of the final output
$\Xhat_N$.  Concretely, for a given accuracy $\varepsilon > 0$, say
that we wish to certify the first term in the
bound~\eqref{EqnNewSingle} is at most $\varepsilon/2$; doing so
requires an iteration number $N \geq \frac{4 \Hinfo(\Tinit, \Tfinal)
  \log(\Tfinal/\Tinit)}{\varepsilon}$.  Consequently, the key
technical challenge is to obtain high-probabability bounds on the \dgc
function $\Hinfo(\Tinit, \Tfinal)$, ideally ones that are tight to
within a constant multiplicative factor.  In this section, we develop
and analyze a procedure for doing so based on samples $\{\Zup{i}
\}_{i=1}^m$.  It exploits the denoising increments and sandwich
property introduced in~\Cref{LemDinfoSandwich}
from~\Cref{SecMartingale}.

\subsubsection{A tail-robust estimator}
\label{SecTailRobust}
For future usability of this result, let us describe the procedure for
estimating $\Hinfo(a,b)$ over an arbitrary interval $[a,b]$.  We then
specialize to the interval $[\Tinit, \Tfinal]$ so as to obtain a
certified version of~\Cref{CorNewSingle}.

\paragraph{Constant-factor approximation:}
We begin by observing that we can obtain a constant-factor
approximation by splitting the interval $[a,b]$ with the dyadic
partition $\dyadic_\ell = \min \{ 2^\ell a, b \}$, for $\ell = 0, 1,
\ldots, \DyaTot$; observe that the number of points is bounded as
$\DyaTot \leq 1 + \log_2(b/a)$.  Moreover, the additivity
property~\eqref{EqnHinfoAdditive} ensures that \mbox{$\Hinfo(a,b) =
  \sum_{\ell=0}^{\DyaTot-1} \Hinfo(\dyadic_\ell, \dyadic_{\ell+1})$.}
Recall the definition~\eqref{EqnDefnDinfo} of the denoising increment
\mbox{$\Dinfo(s, t) \defn \Exs \| \DenTime_s(X_s) - \DenTime_t(X_t)
  \|_2^2$,} as well as the constant-factor approximation
property~\eqref{EqnConstantFactor}.  Combining these ingredients, our
dyadic partition yields the sandwich relation
\begin{align}
  \label{EqnNewSandwich}
  \frac{1}{2} \Hupper(a,b) \; \leq \; \Hinfo(a,b) \; \leq \Hupper(a,b)
  \qquad \mbox{where $\quad \Hupper(a,b) \defn \dfrac{1}{2} \sum
    \limits_{\ell=0}^{\DyaTot-1} \dfrac{\Dinfo(\dyadic_\ell,
      \dyadic_{\ell+1})}{\dyadic_\ell}$.}
\end{align}
which is the population-level basis for our estimators.

\paragraph{Monte Carlo estimates:}  Having reduced our problem
to estimating the function $\Hupper(a,b)$, we first describe how to
construct an unbiased Monte Carlo estimate, followed by a more
sophisticated truncated estimator to handle heavy tails.  Given a
sample $\Zvar \sim \Prob_Z$, consider the forward heat path $X_t =
\Zvar + B_t$, where $\{B_t \}_{t \geq 0}$ is Brownian motion.  We can
evaluate this forward heat path over the dyadic grid by adding
suitable Brownian increments to $\Zvar$, thereby obtaining the vector
$\{\Xvar_{\dyadic_{\ell}} \}_{\ell=0}^{\DyaTot}$.  Using this vector,
we can compute the statistic
\begin{align}
\label{EqnBasicQ}
\Qfun & \defn \frac{1}{2} \sum_{\ell = 0}^{\DyaTot - 1} \frac{ \left\|
  \DenTime_{\dyadic_\ell}(X_{\dyadic_\ell}) - \DenTime_{\dyadic_{\ell
      + 1}}(X_{\dyadic_{\ell + 1}}) \right\|_2^2 }{\dyadic_\ell},
\qquad \mbox{which satisfies $\Exs[\Qfun] = \Hupper(a,b)$ by
  construction.}
\end{align}
Thus, given $\numobs$ i.i.d. samples $\{\Zup{i} \}_{i=1}^\numobs$, we
can compute the statistic $\Qfunup{i}$ for each $\Zup{i}$, and thereby
form the naive Monte Carlo estimate $\frac{1}{m} \sum_{i=1}^\numobs
\Qfunup{i}$.

\paragraph{Tail-robust procedure:}  For tail-robustness,
here we analyze a slightly more sophisticated procedure that involves
truncation.  We assume a known constant $M_p$ such that the denoising
function satisfies the $p^{th}$-moment bound
\begin{align}
\label{EqnMomentCondition}
\left( \Exs \left[ \left\| \DenTime_{t}(X_{t}) \right\|_2^p \right]
\right)^{1/p} & \leq M_p.
\end{align}
For instance, this bound holds trivially with $M_p = R$ if
$\|\Zvar\|_2 \leq R$, but more generally, it allows for unbounded
supports with only polynomial tail decay.
  
Under the condition~\eqref{EqnMomentCondition} and for a given failure
probability $\eta \in (0,1)$, we analyze the truncated estimator
\begin{align}
\label{EqnDefnHhat}  
\Hhat(a,b) & \defn \frac{1}{\numobs} \sum_{i=1}^\numobs \min \{
\Qfunup{i}, \tau \} \qquad \mbox{where $\tau \defn \dfrac{4 M_p^2}{a}
  \left( \dfrac{ 3(p - 2)(\numobs - 1) }{ 14 \log(4 / \eta) }
  \right)^{2/p}$.}
\end{align}

The upper confidence correction~\eqref{EqnDefnHackErr} in our
guarantee involves the \emph{empirical variance estimate}
\begin{align}
\label{EqnDefnEmpBernstein}
  \Vhat \defn \frac{1}{\numobs - 1} \sum_{i = 1}^{\numobs} \left( \min
  \{ \Qfunup{i}, \tau \} - \Hhat(a,b) \right)^2.
\end{align}

\mygraybox{
\begin{proposition}[Data-dependent sandwich on \dgc function]
\label{PropDataSingle}
Under the moment condition~\eqref{EqnMomentCondition}, for any $\eta
\in (0,1)$, the \dgc function $\Hinfo(a,b)$ satisfies the
data-dependent sandwich
\begin{subequations}
\begin{align}
\label{EqnDataSingle}  
\Hinfo(a,b) \; \stackrel{(I)}{\leq} \; \Hhat(a,b) + \HackErr \;
\stackrel{(II)}{\leq} \; 2 \, \big \{ \Hinfo(a,b ) + \HackErr \big \}
\qquad \mbox{with probability at least $1 - \eta$,}
\end{align}
where the upper confidence correction is given by
\begin{align}
  \label{EqnDefnHackErr}
\HackErr & \defn \sqrt{ \frac{ 2 \Vhat \log(4 / \eta) }{\numobs} } +
\frac{8 M_p^2}{a} \left( \frac{ 7 \log(4 / \eta) }{ 3(\numobs - 1) }
\right)^{1 - 2/p}.
\end{align}
\end{subequations}
\end{proposition}
}


\subsubsection{A data-certified sampling guarantee}
\label{SecSingleCertified}
When combined with~\Cref{CorNewSingle}, \Cref{PropDataSingle} has an
immediate practical consequence that is very precise.  For a given
target accuracy $\varepsilon > 0$, suppose that we are given an
initialization $X_T \sim \Qprob_T$ such that $\KL(\Prob_T \|
\Qprob_\Tfinal) \leq \varepsilon/2$.  Here we specify a simple
data-dependent procedure whose output has final KL error at most
$\varepsilon$ with probability at least $1 - \eta$. \\

\mygraybox{\paragraph{Data-dependent procedure with certification:}
  Given i.i.d. samples $\{\Zup{i} \}_{i=1}^\numobs$ and a failure
  probability $\eta \in (0,1)$, we perform the following steps:
\begin{enumerate}
\item[(1)] Using the definitions~\eqref{EqnDefnHhat}
  and~\eqref{EqnDefnHackErr}, compute the estimate $\Hhat(\Tinit,
  \Tfinal)$ and confidence correction $\HackErr$.
\item[(2)] Form the empirical iteration count $\Nhat \defn \left
  \lceil \dfrac{4}{\varepsilon} \; (\Hhat(\Tinit, \Tfinal) + \HackErr)
  \: \log(\Tfinal/\Tinit) \right \rceil$.
\item[(3)] Run the \SIEulerRho procedure from~\Cref{CorNewSingle} with
  iteration budget $\Nhat$, and return the output $\Xhat_{\Nhat} \sim
  \Qprob_\Tinit$.
\end{enumerate}
}

We then have the sequence of bounds
\begin{subequations}
\begin{align}
  \KL( \Prob_\Tinit \| \Qprob_\Tinit) \; \stackrel{(i)}{\leq} \;
  \frac{2 \Hinfo(\Tinit, \Tfinal) \log(\Tfinal/\Tinit) }{\Nhat} +
  \frac{\varepsilon}{2} & \; \stackrel{(ii)}{\leq} \;
  \frac{\varepsilon}{2} \; \frac{\Hinfo(\Tinit, \Tfinal)}{\big[
      \Hhat(\Tinit, \Tfinal) + \HackErr \big]} + \frac{\varepsilon}{2}
  \notag \\
\label{EqnCertBound}  
  & \; \stackrel{(iii)}{\leq} \; \frac{\varepsilon}{2} +
\frac{\varepsilon}{2} \; = \; \varepsilon,
\end{align}
where step (i) follows from the bound~\eqref{EqnNewSingle} with
iterations $\Nhat$, along with the initialization assumption; step
(ii) follows from our choice of $\Nhat$; and the inequality (iii)
holds with probability at least $1 - \eta$, with the bound (I) from
the guarantee~\eqref{EqnDataSingle} in~\Cref{PropDataSingle}.

Moreover, note that the bound (II) from equation~\eqref{EqnDataSingle}
guarantees that the empirical iteration count $\Nhat$ is not much
larger than the oracle requirement $\Nideal = \lceil
\frac{4}{\varepsilon} \Hinfo(\Tinit, \Tfinal) \log(\Tfinal/\Tinit)
\rceil$.  In particular, we have
\begin{align}
\label{EqnNhatBound}
  \Nhat & \leq \left \lceil \dfrac{8}{\varepsilon} \big \{
  \Hinfo(\Tinit, \Tfinal) + \HackErr \big \} \: \log(\Tfinal/\Tinit)
  \right \rceil \; \leq 2 \Nideal + \left \lceil
  \dfrac{8}{\varepsilon} \: \HackErr \log(\Tfinal/\Tinit) \right
  \rceil,
\end{align}
\end{subequations}
so that we have lost at most a factor of two, along with the upper
confidence correction.

\subsubsection{Robustness to learned score functions:}

As noted following~\Cref{ThmMaster}, in practice, diffusion algorithms
are run with estimated score functions $\scorehat_t$, or equivalently
with estimated denoisers $\DenTimeHat_t$.  Recall that our estimator
is built on the sandwich relation~\eqref{EqnNewSandwich} involving
$\Hupper(a,b)$.  Here we describe the extension of this sandwich
relation to the statistic
\begin{subequations}
\begin{align}
\label{EqnModifiedHupper}  
\Hmod(a,b) \defn \dfrac{1}{2} \sum \limits_{\ell=0}^{\DyaTot-1}
\dfrac{\DinfoTil(\dyadic_\ell, \dyadic_{\ell+1})}{\dyadic_\ell}.
\end{align}
This function is based on the modified denoising increment
$\DinfoTil(s, t) \defn \Exs \| \DenTimeHat_s(X_s) - \DenTimeHat_t(X_t)
\|_2^2$ that depends on the estimated denoisers $\DenTimeHat_t$, as
opposed to the exact ones.  Finally, to track the denoiser errors, we
define
\begin{align*}
  \DenError(a, b) & \defn \sum_{\ell = 0}^{\DyaTot}
  \frac{\Exs \| \DenTime_{\dyadic_\ell}(X_{\dyadic_\ell})
    - \DenTimeHat_{\dyadic_\ell}(X_{\dyadic_\ell}) \|_2^2}
  {\dyadic_\ell}.
\end{align*}
\end{subequations}
With these definitions, we can prove the \emph{perturbed sandwich
relation}
\begin{align}
\label{EqnRobustSandwich}
\frac{1}{2} \left( \sqrt{\Hmod(a, b)} - \sqrt{3 \DenError(a, b)}
\right)_+^2 & \leq \Hinfo(a, b) \leq \left( \sqrt{\Hmod(a, b)} +
\sqrt{3 \DenError(a, b)} \right)^2.
\end{align}
This bound is the natural generalization of the sandwich
relation~\eqref{EqnNewSandwich} based on $\Hupper(a, b)$; it reduces to
it in the special case of zero denoiser error.  Recall
that the sandwich relation~\eqref{EqnNewSandwich} was the basis for
the tail-robust Monte Carlo procedure given in~\Cref{PropDataSingle}.
The modified sandwich~\eqref{EqnRobustSandwich} allows us to
pursue this same approach in the learned score setting.
See~\Cref{SecProofEqnRobustSandwich} for the proof.


\subsection{Information-theoretic upper bounds and model-specific consequences}
\label{SecInfoUpper}

Our primary interest---and the result of most practical relevance---is
the data-based bound on the \dgclong function $\Hinfo(\Tinit,
\Tfinal)$ from~\Cref{PropDataSingle}.  For conceptual purposes, on the
other hand, analytical upper bounds can provide useful insight.
Accordingly, in this section, we explore various analytical upper
bounds; derive their consequences for specific model classes; and
discuss connections to related work.

Recalling the equivalent representation~\eqref{EqnHinfoInfoRep} of
$\Hinfo$, we have the sequence of upper bounds
\begin{align}
\label{EqnWeaker}
\Hinfo(\Tinit, T) & \; \stackrel{(i)}{\leq} \; \Info(Z; X_\Tinit) -
\Info(Z; X_T) + \frac{\hfun(T)}{T} \; \stackrel{(ii)}{\leq} \Info(Z;
X_\Tinit) + \frac{\hfun(T)}{T},
\end{align}
where we have used the non-negativity of the MSE function $\hfun$ and
mutual information to drop terms in each step.  The term $\hfun(T)/T$
is straightforward to control, so that the remaining issue is to upper
bound either the \emph{information increment} $\Info(Z; X_\Tinit) -
\Info(Z; X_\Tfinal)$, or simply the \emph{terminal information}
$\Info(Z; X_\Tinit)$.

In the remainder of this section, we derive various upper bounds on
$\Hinfo(\Tinit, \Tfinal)$, either by controlling the information
increment, or the terminal information.  We discuss how these bounds
recover various guarantees on diffusion sampling from past work.

\subsubsection{Covariance-based bounds}

We begin with some simple yet informative upper bounds that depend
only on the covariance structure of the latent variable $\Zvar$, and
allow us to rederive some sampling guarantees from past work in a
direct manner.
\mygraybox{
\begin{proposition}
\label{PropCovariance}
  In terms of the eigenvalues $\{\eigval_j\}_{j=1}^d$ of $\CovZ \defn
  \cov(Z)$, we have the upper bounds
  \begin{align}
\label{EqnCovariance}    
\Hinfo(\Tinit, T) & \leq \frac{1}{2} \sum_{j=1}^d \log \left \{
\frac{1 + (\eigval_j/\Tinit)}{1 + (\eigval_j/T)} \right \} +
\frac{\hfun(T)}{T} \; \leq \; \frac{1}{2} \log \det \Big( \IdMat +
\frac{\CovZ}{\Tinit} \Big) + \frac{\hfun(T)}{T}.
\end{align}
\end{proposition}
}
\noindent This corollary recovers and sharpens a number of results
from past work; let us discuss a few here.

\paragraph{Worst-case linear scaling in dimension:}

Initial results on the iteration complexity of diffusion sampling
established polynomial but superlinear dependence on dimension
(e.g.,~\cite{LeeEtAl22,CheEtAl23a}).  Later work then reduced the
scaling to linear in dimension~\cite{ConEtAl23,BenEtAl24,LiEtAl24},
which is optimal in a worst-case sense.  Here we observe that this
linear scaling in dimension $\usedim$ is an immediate consequence
of~\Cref{PropCovariance}, since
\begin{align}
\label{EqnInterCovBound}
\log \det \Big( \IdMat + \frac{1}{\Tinit} \CovZ \Big) & \;
\stackrel{(i)}{\leq} \; \usedim \: \log \Big(1 + \frac{\Tr
  \CovZ}{\usedim \Tinit} \Big) \; \stackrel{(ii)}{\leq} \; \usedim \:
\log \Big(1 + \frac{\OpNorm{\CovZ}}{\Tinit} \Big),
\end{align}
where step (i) follows from Jensen's inequality; whereas step (ii)
follows since $\Tr \CovZ \leq \usedim \OpNorm{\CovZ}$.

\paragraph{Bounded $R$-models:}
A frequent assumption
(e.g.,~\cite{Bor22,LeeEtAl23,LiEtAl24,BenEtAl24,GatEtAl26}) in
analysis of diffusion-based sampling is that $\Zvar$ is uniformly
bounded, say as $\|\Zvar\|_2 \leq R$ for some radius $R$.  In this
case, we can be fully explicit in terms of the triple $(R, \Tinit,
\usedim)$.  In particular, we have $\Tr \CovZ \; \leq \; \Exs
\|\Zvar\|_2^2 \leq R^2$, and $\hfun(T) \leq \Tr \CovZ \leq R^2$.
Consequently, for any $\Tfinal \geq R^2$, we have the upper bound
$\Hinfo(\Tinit, \Tfinal) \leq 1 + \frac{\usedim}{2} \log \big( 1 +
\frac{R^2}{\usedim \Tinit} \big)$.  This bound exhibits linear scaling
in $\usedim$, and logarithmic scaling in the ratio $R^2/\Tinit$,
consistent with the sharpest results on bounded support
models~\cite{LiEtAl24}.

\paragraph{Dimension-independence:}  If one is willing to assume
that the ratio $R^2/\Tinit$ is bounded independently of dimension,
then our results yield iteration complexity that is
dimension-independent.  In particular, the inequality $\log(1 + t)
\leq t$, applied to the inequality following from step (i), yields the
upper bound $\log \det( \IdMat + \frac{1}{\Tinit} \CovZ ) \leq \Tr
\CovZ / \Tinit \leq R^2/\Tinit$.  Combining the ingredients yields an
iteration complexity \mbox{$\Nscore(\varepsilon) \asymp (R^2/\Tinit)
  \;(1/\varepsilon)$} for sampling in KL divergence.  This should be
  compared with the recent paper~\cite{GatEtAl26}, which analyzes a
  higher-order flow solver with iteration complexity bounded as
  $(R^2/\Tinit) \; \polylog(d, R^2/\Tinit, \varepsilon)$.  By
  comparison, our guarantee for the simpler Euler method fully removes
  the dimension-dependence, at the price of polynomial instead of
  logarithmic scaling in $1/\varepsilon$.

\paragraph{Alternative route to data-based certification:}

The data-based certification procedure in~\Cref{SecSingleCertified}
requires samples of $\Zvar$, and if we also incorporate score-based
errors, these samples must not be used to fit the score functions.
(This hold-out property is required so as to ensure that our Monte
Carlo estimates of the denoising increments are unbiased.)  In
contrast, the bounds from~\Cref{PropCovariance} depend only on the
eigenvalues of the covariance matrix $\CovZ$, which can be directly
estimated from samples $\Zvar$, regardless of whether or not they were
used in the score training process.  The trade-off is that~\Cref{PropDataSingle} targets the finer-grained measure
$\Hinfo(\Tinit, \Tfinal)$, as opposed to the cruder covariance-based bound
from~\Cref{PropCovariance}.


\subsubsection{Rate-distortion and metric entropy}
\label{SecGoodRateDistor}

The covariance-based bound~\eqref{EqnCovariance} is relatively crude,
involving only second-moment structure.  Here we describe a more
refined bound that connects to both rate-distortion theory and
metric entropy.  These bounds are based on a (possibly stochastic)
encoding $\Zhat$, either based on the pair $(\Zvar, X_\Tfinal)$ or on
the variable $\Zvar$ alone.  We represent these encodings by
conditional distributions of the form $\Prob_{\Zhat \mid \Zvar,
  X_\Tfinal}$ and $\Prob_{\Zhat \mid \Zvar}$.  With this notation, we
have the following guarantees:
\mygraybox{
\begin{proposition}[Stochastic encodings and rate-distortion]
\label{PropGoodRateDistor}
We have the upper bound
\begin{subequations}
\begin{align}
  \label{EqnGoodRateDistor}
  \Hinfo(\delta, T) & \leq \inf_{\Prob_{\Zhat \mid \Zvar, X_T}}
  \left\{ \ShanInfo(\Zvar; \Zhat \mid X_T) + \frac{1}{2} \left(
  \frac{1}{\delta} - \frac{1}{T} \right) \Exs \|\Zvar - \Zhat\|_2^2
  \right\} + \frac{\hfun(T)}{T}.
\end{align}
By ignoring $\Xvar_T$, we obtain the weaker upper bound
\begin{align}
  \label{EqnGoodRateDistorWeak}
 \Hinfo(\delta, T) & \leq \inf_{\Prob_{\Zhat \mid \Zvar}} \left\{
 \ShanInfo(\Zvar; \Zhat) + \frac{1}{2 \delta} \Exs \|\Zvar -
 \Zhat\|_2^2 \right\} + \frac{\hfun(T)}{T}.
\end{align}
\end{subequations}
\end{proposition}
}
\noindent See~\Cref{SecProofPropGoodRateDistor} for the proof. \\

These inequalities can be related to the Shannon rate-distortion
functions $\ShanDist_{Z \mid X_T}$ and $\ShanDist_{Z}$, respectively.
Focusing on the second inequality, if we minimize over all encodings
$\Prob_{\Zhat \mid Z}$ such that $\Exs \|\Zhat - Z\|_2^2 \leq u$, then
the minimal value of $\ShanInfo(\Zhat; \Zvar)$ is equal to the rate
distortion function $\ShanDist_\Zvar(u)$.  Since this argument applies
for any $u \geq 0$, we find that $\Hinfo(\delta, T) \leq \inf_{u \geq
  0} \Big \{ \ShanDist_\Zvar(u) + \frac{u}{2 \Tinit} \Big \} +
\frac{\hfun(\Tfinal)}{\Tfinal}$.


\paragraph{Metric entropy bounds:}

When $\Zvar$ has compact support $\Zset$, we can
use~\Cref{PropGoodRateDistor} to derive bounds involving the
\emph{metric entropy} $\log M_2(t; \Zset)$, defined to be the
(logarithm of the) smallest number of Euclidean balls of radius $t$
whose union covers $\Zset$.  (See the standard sources~\cite{Wai19a}
for further background on metric entropy and its properties.)  For any
error $t > 0$, fix a covering $\{z^1, \ldots, z^M \}$ of $\Zset$ with
$M = M_2(t; \Zset)$ elements, and then define the encoding $\Zhat =
\arg \min_{j = 1, \ldots, M} \|Z - z^j\|_2$.  By construction, we have
$\Exs \|\Zhat - \Zvar\|_2^2 \leq t^2$, and since $\Zhat$ is a discrete
random variable taking at most $M$ values, we have $\ShanInfo(\Zvar;
\Zhat) \leq \Ent(\Zhat) \leq \log M$. Applying the
bound~\eqref{EqnGoodRateDistorWeak} with these ingredients, we find
that
\begin{subequations}
\begin{align}
  \label{EqnMetricBound} 
  \Hinfo(\delta, T) & \leq \inf_{t > 0} \Big \{ \log M_2(t; \Zset) +
  \frac{t^2}{2 \delta} \big \} + \frac{\hfun(T)}{T}.
\end{align}
This bound is useful because it connects directly to the metric
entropy literature, but it can be loose, since the bound $\Ent(\Zhat)
\leq \log M$ ignores any probabilistic structure present in the
encoding $\Zhat$.

\paragraph{Consequences for metric effective dimension:}
Let us say that a set $\Zset$ has polynomial metric entropy with
\emph{metric effective dimension} $\dmin$ if $\log M_2(t; \Zset)
\leq c_0 + \dmin \log \big(1 + \frac{c_1}{t} \big)$ for all $t \in
(0,1]$, where $(c_0, c_1)$ are universal constants.  Assuming
$\delta \leq 1$ and setting $t = \sqrt{\delta}$ in the
bound~\eqref{EqnMetricBound}, we find that
\begin{align*}
  \Hinfo(\Tinit, \Tfinal)
\leq 1 + c_0 + \dmin \log \Big(1 + \frac{c_1}{\sqrt{\Tinit}} \Big) +
\frac{\hfun(\Tfinal)}{\Tfinal}.
\end{align*}
Combined with our general sampling guarantee, this yields iteration
complexity that is linear in the metric effective dimension $\dmin$.
This matches and removes the logarithmic slack from the
intrinsic-dimension dependence established in recent DDPM
analyses~\cite{LiYan24,HuaEtAl24a}.  Since a compact
$\mdim$-dimensional regular manifold has metric dimension $\mdim$, it
matches results~\cite{PotEtAl24} for manifold-supported distributions,
while applying under the broader metric-entropy condition in the
paper~\cite{HuaEtAl24a}.

\paragraph{Consequences for Gaussian mixtures:}

Recent work has established rigorous guarantees for learning and
sampling Gaussian mixture models using diffusion methods
(e.g.,~\cite{ShaEtAl23,GatEtAl24,CheKonSha25,LiCaiWei26}).
\Cref{CorNewSingle} and~\Cref{PropGoodRateDistor} have concrete
implications for such models.  More precisely, suppose that $\Zvar$ is
a discrete random variable taking values in the subset $\{\mu_1,
\ldots, \mu_K \}$ for a collection of centers such that $\|\mu_k\|_2
\leq R$ for each $k \in \{1, \ldots, K\}$.  The distribution
$\Prob_\Tinit$ is then a $K$-component Gaussian mixture.  Setting
$\Zhat = Z$ in the bound~\eqref{EqnGoodRateDistorWeak}, we find that
$\Hinfo(\Tinit, \Tfinal) \leq \Info(Z; Z) + \frac{\hfun(T)}{T} =
\Ent(Z) + \frac{\hfun(T)}{T}$, where $\Ent$ is the Shannon entropy. As
a consequence, using the iteration complexity
bound~\eqref{EqnSingleBlockStrong} with $T = 2R^2$, we can sample to
$\varepsilon$ accuracy from the Gaussian mixture $\Prob_\Tinit$ using
at most
\begin{align}
\label{EqnGaussMixNscore}
\NscoreTot(\varepsilon) & \leq 2 + \log(2R^2/\Tinit) +
\frac{4}{\varepsilon} \big \{ \Ent(Z) + 1 \big \} \; \log(2
R^2/\Tinit)
\end{align}
\end{subequations}
score evaluations.  Note that $\Ent(Z) \leq \log K$, with equality
achieved when $Z$ has a uniform distribution, so that the
bound~\eqref{EqnGaussMixNscore} gives dimension-independent scaling
with logarithmic scaling in $K$, as in the paper~\cite{LiCaiWei26}.
This analysis can also be extended to the setting in which the
distribution is not a Gaussian mixture, but can be well-approximated
by one in the $\Wass_2$-distance; see~\Cref{SecApproxStructureMix} for
details.  This leads to approximation-theoretic guarantees that are
related to but distinct from the TV guarantees given in the
paper~\cite{LiCaiWei26}.


\subsubsection{Logarithmic dependence on Poincar\'{e} constant}

It is well known that various geometric properties of the target
distribution $\Prob_Z$ affect the iteration complexity of sampling.
The most restrictive property is that of strong log-concavity, which
can be substantially weakened via the notion of log-Sobolev or
\Poincare constants; see the book~\cite{Che25} for relevant
background.  Here we consider the \Poincare condition, which is the
weakest of the three.  We say that $\Zvar$ has \emph{\Poincare
constant} $\Pcar$ if
\begin{subequations}
\begin{align}
\label{EqnDefnPoincare}
\var(g(\Zvar)) & \leq \Pcar \Exs \|\nabla g(\Zvar)\|_2^2
\end{align}
for all differentiable functions $g: \real^d \rightarrow \real$.  In
the analysis here, we also assume that the density $p_\Zvar$ is
differentiable and satisfies a \emph{one-sided $L$-smoothness
condition}, meaning that $f(z) \defn - \log p_\Zvar(z)$ satisfies the
upper bound
\begin{align}
  \label{EqnLsmooth}
f(z + u) & \leq f(z) + \inprod{\nabla f(z)}{u} + \tfrac{L}{2} \|u\|_2^2
\qquad \mbox{for all $z, u \in \real^d$.}
\end{align}
\end{subequations}
The following result applies to a centered variable ($\Exs[\Zvar] =
0$), with $\Tinit \defn \frac{\varepsilon}{8 L \usedim}$, and \SIEuler
initialization $\Xhat_0 \sim \Normal(0, \Tfinal \IdMat)$ where $\Tfinal
\defn \Tinit + \frac{2 \Pcar \usedim}{\varepsilon}$.
\mygraybox{
  \begin{proposition}[Logarithmic dependence on \Poincare constant]
\label{PropPoincare}      
Under conditions~\eqref{EqnDefnPoincare} and~\eqref{EqnLsmooth} and
the specified initialization for any $\varepsilon \in (0,1)$, running
the \SIEuler scheme yields samples $\Xtil \sim \Qprob$ such that
\mbox{$\KL( \Prob_\Zvar \| \Qprob) \leq \varepsilon$} with iteration
complexity at most
\begin{align}
 \label{EqnPoincare}
\NscoreTot(\varepsilon) & \leq 1 + \frac{4}{\varepsilon} \; \Big \{
\usedim \log \big (1 + \frac{8 \usedim \kappa}{\varepsilon}) + 1 \Big
\} \log \left(1 + \frac{16 \kappa \usedim^2}{\varepsilon^2} \right),
\end{align}
where $\kappa \defn L \: \Pcar$ is the effective condition number.
\end{proposition}

}
\noindent See~\Cref{SecProofPropPoincare} for the proof. \\

This bound should be compared with our own recent work~\cite{Wai26},
which uses the diffusion path to reduce to a sequence of strongly
log-concave (SLC) problems, and proved (cf. Thm. 1 in the
paper~\cite{Wai26}) the iteration complexity bound
\begin{subequations}
\begin{align}
\label{EqnWainwrightGeneral}  
  \Nwainwright(\varepsilon) \leq \big \{ 2 + \log(\kappa) \big \} \;
  \NSLC \big(\tfrac{\varepsilon}{2 + \log (\kappa)} \big),
\end{align}
where $\kappa \defn L/m$ is the condition number of the problem.  Here
$\NSLC(u)$ denotes the iteration complexity of any procedure for
sampling from an SLC distribution with constant conditioning.  If, for
example, we instantiate this general guarantee with $\NSLC$ chosen to
be the complexity of unadjusted Langevin (ULA), then we obtain a final
guarantee of
\begin{align}
\label{EqnWainwright}
\Nwainwright(\varepsilon) & \leq \frac{c}{\varepsilon} \usedim (2 +
\log (\kappa))^2
\end{align}
\end{subequations}
for a universal constant $c$.  To compare with the
guarantee~\eqref{EqnPoincare}, we note that any $m$-strongly
log-concave distribution has \Poincare constant $\Pcar \leq 1/m$, so
that $L \Pcar \leq \kappa = L/m$.  Thus, for the ULA sampler,
\Cref{PropPoincare} recovers the guarantee~\eqref{EqnWainwright} up to
poly-logarithmic factors in $(\usedim, 1/\varepsilon)$.  If the
general bound in~\Cref{ThmMaster} is instantiated with different SLC
samplers, then stronger guarantees can be obtained: for instance, with
the scaling $\sqrt{\usedim} \polylog(1/\varepsilon)$ when using
Metropolis-type samplers, or $d^{1/3}/\epsilon^{2/3}$ when using
randomized midpoint methods.  That being said, \Cref{PropPoincare}
applies to the far broader problem class of distributions with finite
\Poincare constant.

\subsubsection{Relation to a prior single-block bound}

As previously noted, Reeves and Pfister~\cite{ReePfi25} used I--MMSE
identities and information-theoretic methods to control the error of
KL discretization, using a single-block approach.  Using our
notation, their analysis (cf. Thm. 2 and Cor. 2 in their paper) with
optimized stepsize parameter leads to iteration complexity
proportional to the terminal mutual information $\Info(Z; X_\delta)$.
Note that the \dgclong $\Hinfo(\delta, T)$ is upper bounded by
$\Info(Z; X_\delta)$ up to a terminal correction;
cf. inequality~\eqref{EqnWeaker}(ii).  Thus, the single-block
guarantees developed in this section are related to their result, but
arise here as a very special case of the arbitrary-grid bound
in~\Cref{ThmMaster}. As opposed to a single-block approach, the
general form of~\Cref{ThmMaster} decomposes the full path into local
increments $\Hinfo(t_{\iter+1}, t_\iter)$. This interval-local form
becomes essential in the next section, where different portions of the
diffusion path are assigned different geometric resolutions, thereby
leading to a procedure with strictly tighter performance guarantees.


\section{Optimal bounds for tracking the \dgc path}
\label{SecTracking}

Thus far, we discussed a very simple consequence of~\Cref{ThmMaster},
involving a geometric multiplier $\rho$ over the full interval
$[\Tinit, \Tfinal]$.  However, the \dgc function is additive across
intervals~\eqref{EqnHinfoAdditive}, which enables us to pursue a more
refined analysis.  More precisely, we can obtain sharper guarantees by
partitioning the interval $[\Tinit, \Tfinal]$ into a sequence of
blocks, and then assigning a different geometric multiplier to each
block.  Doing so allows the algorithm to adapt to the structure of the
\dgc path.  In this section, we develop this strategy via a sequence
of refinements.  We start in~\Cref{SecTrackingMulti} by fixing a
collection of blocks, and showing how to optimize the choices of
geometric multipliers.  This optimization leads to a functional of the
partition~\eqref{EqnDGCPartComplex} that we refer to as the
\emph{\dgc-based partition complexity}. In~\Cref{SecDataMultiple}, we
exploit the tail-robust estimator from~\Cref{SecDataSingle} to show
how this optimization can be performed in a data-dependent manner.
Finally, in~\Cref{SecTrackingBlock}, we discuss how dynamic
programming can be used to select optimal block boundaries.  Overall,
these results allow us to achieve bounds based on an increasingly
refined tracking of the \dgc path.


\subsection{Bounds using optimal $\Btot$-block schemes}
\label{SecTrackingMulti}

We begin with the simplest multi-block extension.  Fix a $\Btot$-block
partition of the interval $[\Tinit, \Tfinal]$, say of the form $\Tinit
= \block_0 < \block_1 < \cdots < \block_\Btot = \Tfinal$, yielding the
collection of blocks
\begin{subequations}
\begin{align}
\label{EqnDefnPartition}  
\Partition & \defn \{ [\block_{\bind}, \block_{\bind+1}] \mid \bind =
0, \ldots, \Btot - 1 \}.
\end{align}
Associated with each block are the \emph{log-heat-time block length}
$\Sinfo_\bind \defn \log\left( \frac{\block_{\bind + 1}}{\block_\bind}
\right)$ and the \emph{\dgc block length} $\Hinfo_\bind \defn
\Hinfo(\block_\bind, \block_{\bind + 1})$.\label{EqnDefnSinfo}  Given a total iteration
budget $N$, suppose that we allocate $N_\bind$ iterations to each
$[\block_{\bind+1}, \block_\bind]$, subject to the constraint
$\sum_{\bind=0}^{\Btot-1} N_\bind = N$.  Applying~\Cref{ThmMaster}
with this block partition then ensures that
\begin{align}
\label{EqnMultiUpper}  
\KL( \Prob_\Tinit \| \Qprob_\Tinit) & \leq \sum_{\bind = 0}^{\Btot-1}
\rho_\bind \Hinfo_\bind + \KL( \Prob_\Tfinal \| \Qprob_\Tfinal) \qquad
\mbox{where $\rho_\bind \defn \exp \{\Sinfo_\bind/N_\bind \} - 1$.}
\end{align}
Combining the ingredients, we arrive at the following
integer-programming problem
\begin{align}
\label{EqnKblockDP}
  \min_{(N_0, \ldots, N_{\Btot-1})} \Big \{ \sum_{\bind = 0}^{\Btot-1}
  \big[ \exp \{\Sinfo_\bind/N_\bind \} - 1 \big] \Hinfo_\bind \Big \}
  \qquad \mbox{subject to the constraint $\sum
    \limits_{\bind=0}^{\Btot-1} N_\bind = N$.}
\end{align}
\end{subequations}
It is easy to see that this problem can be solved via dynamic
programming; the resulting solution allows us to achieve the optimal
guarantee afforded by~\Cref{ThmMaster} for the given $K$-block
partition.  Moreover, as the partition is refined by increasing
$\Btot$, we can argue that (under mild regularity conditions on
$\hfun'$), this optimal guarantee should approach the optimal KL
discretization error; see~\Cref{SecLocalSharpness} for the relevant
discussion.

\subsubsection{Guarantees for an explicit procedure}

While an explicit dynamic programming procedure is feasible, it is
helpful---both for theoretical understanding and for the data-driven
procedure developed in~\Cref{SecDataMultiple}---to describe a simple
and completely explicit allocation that yields an order-optimal
solution to the dynamic program~\eqref{EqnKblockDP}.  Moreover, this
result highlights the relevance of the \emph{\dgc-based partition
complexity} defined by
\begin{align}
  \label{EqnDGCPartComplex}
  \PartH & \defn \left( \sum_{\bind = 0}^{\Btot - 1}
  \sqrt{\Sinfo_\bind \Hblk_\bind} \right)^2.
\end{align}

\mygraybox{
\begin{theorem}[Optimal geometric schedules for $\Btot$-block partition]
\label{ThmTrackMulti}
Given the $\Btot$-block partition $\Partition $ and an iteration
budget $\Nscore \geq 2 \left\{ K + 2 \log(\Tfinal/\Tinit) \right\}$,
define the geometric multipliers
\begin{subequations}
\begin{align}
  \label{EqnFixedBudgetRho}
  \rho_\bind & \defn \min\left\{ 1, 4 \frac{\sqrt{\PartH}}{\Nscore}
  \sqrt{\frac{\Sinfo_\bind}{\Hblk_\bind}} \right\} \quad \mbox{for $\bind
    = 0, 1, \ldots, \Btot-1$.}
\end{align}
Then the $\Btot$-block \SIEulerRho procedure based on partition
$\Partition$ uses at most $\Nscore$ score evaluations, and yields
output $\Xhat_\Nscore \sim \Qprob_\Tinit$ such that
\begin{align}
\label{EqnFixedBudgetKL}
\KL(\Prob_\Tinit \| \Qprob_\Tinit) & \leq \frac{4 \PartH}{\Nscore} +
\KL(\Prob_\Tfinal \| \Qprob_\Tfinal).
\end{align}
Moreover, the optimal DP solution~\eqref{EqnKblockDP} based on
iteration budget $N$ satisfies the lower bound
\begin{align}
\label{EqnDGCLower}  
\sum_{\bind = 0}^{\Btot-1} \rho_\bind \Hinfo_\bind & \geq
\frac{\PartH}{N}.
\end{align}
\end{subequations}
\end{theorem}
}
\noindent See~\Cref{SecProofThmTrackMulti} for the proof of this
claim.

\paragraph{Comparison to the DP optimum:}
Ignoring the terminal cost $\KL(\Prob_\Tfinal \| \Qprob_\Tfinal)$,
note that the upper bound~\eqref{EqnFixedBudgetKL} and lower
bound~\eqref{EqnDGCLower} ensure that we have obtained a solution that
is within a factor of $4$ of the optimal value of the dynamic
programming formulation~\eqref{EqnKblockDP}.

\paragraph{Comparison to the single-block complexity:}
From the inequality~\eqref{EqnFixedBudgetKL}, we see that iteration
complexity scales linearly in the partition complexity $\PartH$ from
equation~\eqref{EqnDGCPartComplex}. It is worth comparing this with
the single-block guarantees from~\Cref{CorNewSingle}.  Observe that if
we use the partition $\Partition_1 = \{ [\Tinit, \Tfinal] \}$ with a
single block, we have
\begin{subequations}
\begin{align*}
\PartComp([\Tinit, \Tfinal]) & = \log (\Tfinal/\Tinit) \Hinfo(\Tinit,
\Tfinal),
\end{align*}
which matches the complexity appearing in the
bound~\eqref{EqnNewSingle} from~\Cref{CorNewSingle}.  Moreover,
observe that for any partition $\Partition$, we have
\begin{align}
\label{EqnCauchy}
\PartH =\left( \sum_{\bind = 0}^{\Btot - 1} \sqrt{\Sinfo_\bind
  \Hblk_\bind} \right)^2 & \stackrel{(i)}{\leq}\Big(
\sum_{\bind=0}^{\Btot-1} \Sinfo_\bind \Big) \; \Big(
\sum_{\bind=0}^{\Btot-1} \Hblk_\bind \Big) \; \stackrel{(ii)}{=} \;
\PartComp([\Tinit, \Tfinal]),
\end{align}
\end{subequations}
where step (i) uses the Cauchy--Schwarz inequality; and step (ii) uses
the relation $\sum_{\bind=0}^{\Btot-1} {\Sinfo_\bind} =
\log(\Tfinal/\Tinit)$, and $\sum_{\bind=0}^{\Btot-1} \Hblk_\bind =
\Hinfo(\Tinit, \Tfinal)$, where the latter property follows from the
additivity~\eqref{EqnHinfoAdditive} of the \dgc function $\Hinfo$.
Thus, we see that the multi-block complexity is never larger than the
single-block complexity, as should be expected.  More importantly,
this argument illustrates when the multi-block approach yields gains
over the single-block approach.  The Cauchy--Schwarz step is the only
source of looseness, and it holds with equality if and only if there
is some scalar $c > 0$ such that $\Sinfo_\bind = c \Hinfo_\bind$ for
each block.  This is a highly restrictive condition, and
in~\Cref{SecFine}, we explore the nature of this gap in more detail.
In particular, \Cref{SecSeparation} describes a very simple ensemble
of Gaussian mixtures for which the multi-block and single-block
complexities exhibit an exponential separation in terms of the log
radius $R$.


\subsubsection{Certified guarantees for multi-block schemes}
\label{SecDataMultiple}

\Cref{ThmTrackMulti} provides an attractive and near-optimal upper
bound, but observe that the stepsize choices~\eqref{EqnFixedBudgetRho}
depend on knowledge of the \dgc increments $\Hinfo_\bind$, as well as
the partition complexity $\PartComp(\Partition) = \big(
\sum_{\bind=0}^{\Btot-1} \sqrt{\Sinfo_\bind \Hinfo_\bind} \big)^2$.
The increments $\Sinfo_\bind$ are known quantities, so that it only
remains to bound $\Hinfo_\bind$.  In~\Cref{SecDataSingle}, we
described a data-based procedure to bound the \dgc increment
$\Hinfo(a,b)$ for any interval $[a,b]$.  Here we show how the
estimator from~\Cref{PropDataSingle}, when combined with the arguments
underlying the proof of~\Cref{ThmTrackMulti}, can be used to construct
a multi-block procedure with certified guarantees.

Suppose that, for each block $[\block_{\bind}, \block_{\bind+1}]$, we
apply the tail-robust estimator from~\Cref{PropDataSingle} with
failure probability $\eta/\Btot$.  It provides estimates $\Hhat_\bind
\defn \Hhat(\block_{\bind}, \block_{\bind+1})$ such that
\begin{subequations}
  \begin{align}
\label{EqnHhatBound}   
\underbrace{ \Hinfo(\block_{\bind}, \block_{\bind+1})}_{\equiv
  \Hinfo_\bind} & \leq \underbrace{\Hhat(\block_{\bind},
  \block_{\bind+1})}_{\equiv \Hhat_\bind} +
\underbrace{\NewHackErr}_{\equiv \rhat_k}.
\end{align}
  Our choice of failure probability $\eta/\Btot$, the union bound, and
  the guarantee from~\Cref{PropDataSingle} ensure that these bounds
  hold uniformly over all blocks with probability at least $1 - \eta$.
  Here the upper confidence bound is given by
\begin{align}
\label{EqnNewHackErr}
\NewHackErr & \defn \sqrt{ \frac{ 2 \Vhat_\bind \log(4 \Btot / \eta) }{\numobs}
} + \frac{8 M_p^2}{\block_{\bind}} \left( \frac{ 7 \log(4 \Btot /
  \eta) }{ 3(\numobs - 1) } \right)^{1 - 2/p},
\end{align}
where $\Vhat_k$ is the empirical variance estimate for the block.
Finally, we define the surrogate partition complexity
\begin{align}
  \label{EqnDefnPartCompHat}
\PartCompHat(\Partition) & \defn \Big( \sum_{\bind =0}^{\Btot-1}
\sqrt{\Sinfo_k} \sqrt{\Hhat_k + \rhat_k} \Big)^2.
\end{align}
\end{subequations}
With this set-up, we have the following \emph{certified analogue}
of~\Cref{ThmTrackMulti}.

\mygraybox{
  \begin{corollary}[Data-certified guarantees for multi-block scheme]
\label{CorCertifiedMulti}    
Given a failure probability $\eta \in (0, 1)$, a $\Btot$-block
partition and an iteration budget $\Nscore \geq 2 \left\{ K + 2
\log(\Tfinal/\Tinit) \right\}$, suppose that we run the \SIEulerRho
scheme with the geometric multipliers
\begin{subequations}
\begin{align}
    \label{EqnCertifiedMultiRho}
    \rhohat_\bind & \defn \min\left\{ 1, \frac{4
      \sqrt{\PartCompHat(\Partition)}}{\Nscore}
    \sqrt{\frac{\Sinfo_\bind}{\Hhat_\bind + \rhat_\bind}} \right\}
    \qquad \mbox{for $\bind = 0, 1, \ldots, \Btot - 1$.}
  \end{align}
Then the output $\Xhat_N \sim \Qprob_\Tinit$ satisfies the bound
\begin{align}
    \label{EqnCertifiedMulti}
    \KL(\Prob_\Tinit \| \Qprob_\Tinit) & \leq \frac{4
      \PartCompHat(\Partition)}{\Nscore} + \KL(\Prob_\Tfinal \|
    \Qprob_\Tfinal) \qquad \mbox{with probability at least $1 -
      \eta$.}
  \end{align}
\end{subequations}
\end{corollary}
}
\noindent We provide the proof here, since it is elementary, and
illustrates a natural robustness of the geometric multiplier scheme to
perturbations.
\begin{proof}
By applying the upper bound~\eqref{EqnMultiUpper} with the geometric
multipliers $\{\rhohat_\bind \}_{\bind=0}^{\Btot-1}$ from
equation~\eqref{EqnCertifiedMultiRho}, we find that
\begin{align*}
\KL( \Prob_\Tinit \| \Qprob_\Tinit) \leq \sum_{\bind = 0}^{\Btot-1}
\rhohat_\bind \Hinfo_\bind + \KL( \Prob_\Tfinal \| \Qprob_\Tfinal) &
\stackrel{(i)}{\leq} \sum_{\bind = 0}^{\Btot-1} \rhohat_\bind
\big(\Hhat_\bind + \rhat_\bind \big) + \KL( \Prob_\Tfinal \|
\Qprob_\Tfinal),
\end{align*}
where inequality (i) holds with probability at least $1 - \eta$, using
each of the block bounds~\eqref{EqnHhatBound}.  From our
choice~\eqref{EqnCertifiedMultiRho} of geometric multipliers, we have
\begin{align*}
  \sum_{\bind = 0}^{\Btot-1} \rhohat_\bind \big(\Hhat_\bind +
  \rhat_\bind \big) \; \leq \; 4 \sum_{\bind = 0}^{\Btot-1} \left \{
  \frac{\sqrt{\PartCompHat(\Partition)}}{\Nscore}
  \sqrt{\frac{\Sinfo_\bind}{\Hhat_\bind + \rhat_\bind}} \right\} \;
  \big(\Hhat_\bind + \rhat_\bind \big) & = 4 \sum_{\bind =
    0}^{\Btot-1} \frac{\sqrt{\PartCompHat(\Partition)}}{\Nscore}
  \sqrt{\Sinfo_\bind (\Hhat_\bind + \rhat_\bind)} \\
  & = 4 \frac{\PartCompHat(\Partition)}{\Nscore}
\end{align*}
where the last step uses the definition~\eqref{EqnDefnPartCompHat} of
the surrogate complexity.

Finally, the same argument used to prove~\Cref{ThmTrackMulti} can be
used to show that the overall procedure uses at most $N$ score
evaluations, thereby meeting the budget.
\end{proof}


\subsubsection{Optimal choice of block boundaries}
\label{SecTrackingBlock}

Our analysis thus far has focused on optimal schemes for a fixed set
of $\Btot$ blocks. In this section, we formulate the problem of
optimizing the split points that define a (near)-optimal set of
$\Btot$ blocks via a simple dynamic program.

More precisely, for an integer $\Jtot > \Btot$ that defines the
resolution, suppose that we fix a fine deterministic grid of
\emph{candidate heat times}
$\tau_0 = \Tinit < \tau_1 < \cdots < \tau_\Jtot = \Tfinal$.
Our goal is to select indices $\{ i_\bind \}_{\bind=1}^{\Btot-1}$
from the set $\{ 1, \ldots, \Jtot - 1 \}$, and use the corresponding
heat times $t_\bind = \tau_{i_\bind}$, together with the endpoints
$t_0 = \Tinit$ and $t_\Btot = \Tfinal$, to define a $\Btot$-block
partition. We seek to minimize the square root of the partition
complexity
\begin{align*}
  \sqrt{\PartComp(\Partition)} & = \sum_{\bind=0}^{\Btot-1}
  \sqrt{\Sinfo_\bind \Hblk_\bind}, \qquad \mbox{where $\Sinfo_\bind
    \defn \log(t_{\bind+1}/t_\bind)$, and $\Hblk_\bind \defn
    \Hinfo(t_\bind, t_{\bind+1})$.}
\end{align*}

Given the naturally sequential structure, this optimization problem
can be solved by dynamic programming, as we now formalize. For
candidate intervals $0 \leq i < j \leq \Jtot$, define the edge cost
\begin{subequations}
\begin{align}
\label{EqnKBlockEdgeCost}
e(i, j) & \defn \sqrt{\Sinfo(\tau_i, \tau_j) \Hinfo(\tau_i, \tau_j)},
\qquad \mbox{where $\Sinfo(a, b) \defn \log(b/a)$.}
\end{align}
With this notation, our ultimate goal is to compute
\begin{align}
  \label{EqnFinalDP}
  \ValFun_\Btot(\Jtot) & \defn \min_{0 = j_0 < j_1 < \cdots < j_\Btot
    = \Jtot} \sum_{\bind=0}^{\Btot-1} e(j_\bind, j_{\bind+1}).
\end{align}
\end{subequations}
This quantity is the minimum cost of a $\Btot$-block partition of the
interval $[\tau_0, \tau_\Jtot] = [\Tinit, \Tfinal]$. The dynamic
program involves the intermediate quantities $\ValFun_\bind(j)$,
defined for each $\bind \in \{ 0, \ldots, \Btot \}$ and $j \in \{ 0,
\ldots, \Jtot \}$ as the minimum cost of a $\bind$-block partition of
the subinterval $[\tau_0, \tau_j]$, with value $+\infty$ when no such
partition exists.

In order to compute the optimal value~\eqref{EqnFinalDP}, we perform
the following forward recursion over stages \mbox{ $\bind = 0, \ldots,
  \Btot - 1$:}
\mygraybox{\paragraph{Dynamic program for selecting an optimal
    $\Btot$-block partition.}
\begin{subequations}
\begin{enumerate}
\item[(1)] At stage $\bind = 0$, initialize the value function
$\ValFun_0$ as
\begin{align}
  \label{EqnKBlockDPInitialization}
  \ValFun_0(j)
  & =
  \begin{cases}
    0, & \mbox{if $j = 0$,} \\
    +\infty, & \mbox{if $j \in \{ 1, \ldots, \Jtot \}$.}
  \end{cases}
\end{align}
\item[(2)] For stages $\bind = 0, \ldots, \Btot - 1$, perform the
updates
\begin{align}
  \label{EqnKBlockDPRecursion}
  \ValFun_{\bind+1}(j)
  & = \min_{i \in \{ \bind, \ldots, j - 1 \}}
  \left\{
    \ValFun_\bind(i) + e(i, j)
  \right\}
  \qquad
  \mbox{for all $j \in \{ \bind + 1, \ldots, \Jtot \}$,}
\end{align}
and set $\ValFun_{\bind+1}(j) = +\infty$ for
$j \in \{ 0, \ldots, \bind \}$.
\end{enumerate}
\end{subequations}
}
\noindent The standard DP reasoning can be used to show that upon
termination, the value $\ValFun_{\Btot}(\Jtot)$ satisfies the required
identity~\eqref{EqnFinalDP}. Direct evaluation of the dynamic program
costs $O(\Btot \Jtot^2)$ arithmetic operations, where $\Jtot$ controls
the discretization accuracy (relative to the idealized continuous time
DP).

Finally, we observe that the edge costs~\eqref{EqnKBlockEdgeCost}
depend on the \dgc increments $\Hinfo(\tau_i, \tau_j)$.  As in the
previous section, we can obtain confidence bounds for increments of
this type using the estimator from~\Cref{PropDataSingle}, thereby
obtaining a data-dependent version of the dynamic program.


\subsection{Log-time \dgc density and the partition-complexity sandwich}
\label{SecFine}

This section is devoted to a deeper exploration of the gaps between
the multi-block guarantees from~\Cref{ThmTrackMulti} and the
single-block guarantee from~\Cref{CorNewSingle} over a given interval
$[\Tinit, \Tfinal]$.  In particular, by considering the fine partition
limit, we can understand these gaps in an integral form, which
involves a reparameterization into log heat-time.

Recall the definition~\eqref{EqnDefnOrigQdens} of the log-time \dgc
  density: here we make direct use of the equivalent form
\begin{subequations}
  \begin{align}
    \label{EqnDefnDGCDensity}
  \qdens(r) & \defn \hfun' \big( \Tinit e^r \big).
  \end{align}
See~\Cref{FigQplots} in~\Cref{SecLogTimeDGC} for some illustrative
plots.  By construction, we have the representation $\Hinfo( \Tinit
e^u, \Tinit e^v) = (1/2) \int_u^v \qdens(r) dr$.  As a particular
case, for the choices $u = 0$ and $v = \log(\Tfinal/\Tinit)$, we have
the equivalence
\begin{align*}
 \Hinfo(\Tinit, \Tfinal) & = (1/2) \int_0^{\log(\Tfinal/\Tinit)}
 \qdens(r) dr.
\end{align*}
\end{subequations}
This expression shows that $\qdens$ corresponds to the amount of \dgc mass
that is accumulated per unit of log-heat-time, hence the name log-time
\dgc density.

\subsubsection{The partition-complexity sandwich}
\label{SecSandwich}

In this section, we show that the partition-complexity $\PartH$ can be
sandwiched between two integrals based on the density $\qdens$.  The
upper bound corresponds to the single-block
complexity~\eqref{EqnNewSingle}, whereas the lower bound is defined by
the \emph{log-time \dgc spread}
\begin{align}
  \label{EqnDGCSpread}
\PartHfine(\Tinit, \Tfinal) & \defn \frac{1}{2} \left(
\int_0^{\log(\Tfinal/\Tinit)} \sqrt{\qdens(r)} \, dr \right)^2.
\end{align}
Recall that the single-block complexity takes the form
$\PartComp([\Tinit, \Tfinal]) = \Hinfo(\Tinit, \Tfinal) \;
\log(\Tfinal/\Tinit)$.  The following result shows that the partition
complexity is always sandwiched between the two extremes.

\mygraybox{
\begin{lemma}[Partition-complexity sandwich]
\label{LemFinePartSandwich}
Suppose that $\qdens$ is integrable.  Then for any partition
$\Partition$, the partition complexity $\PartH$ is sandwiched as
\begin{align}
  \label{EqnFinePartSandwich}
  \underbrace{\PartHfine(\Tinit, \Tfinal)}_{\mbox{\small \dgc spread}}
\; \leq \; \PartH \; \stackrel{(ii)}{\leq} \; \underbrace{\PartComp([\Tinit,
    \Tfinal])}_{\mbox{Single block}}
\end{align}
\end{lemma}
}
\paragraph{Achieving the \dgc spread:}  Our 
proof below also establishes that
\begin{align}
  \label{EqnFineLimit}
\PartHfine(\Tinit, \Tfinal) & = \inf_{\Partition} \PartH
\end{align}
where the infimum is taken over all partitions.  Thus, by choosing
successively finer partitions, we can make $\PartH$ arbitrarily close
to the \dgc spread~\eqref{EqnDGCSpread}. \\

\noindent We give the proof of~\Cref{LemFinePartSandwich} here, since
it is simple and illustrates exactly how the $\Btot$-block procedure
leads to a $\Btot$-piece upper bound on the integral that defines
$\PartHfine(\Tinit, \Tfinal)$.

\begin{proof}
\noindent The bound (ii) follows via the Cauchy--Schwarz inequality;
see equation~\eqref{EqnCauchy} for our earlier argument.  Accordingly,
we now prove the bound (i).  Due to the change of variables involved
in defining $\qdens$, we define partitions of the interval $[0,
  \log(\Tfinal/\Tinit)]$ in log-time as $0 = r_0 < r_1 < \cdots < r_K
= \log(\Tfinal/\Tinit)$.  In terms of this parameterization, the pair
$(\Sinfo_k, \Hinfo_k)$ from equation~\eqref{EqnDefnSinfo} then has
the equivalent (and very convenient) representation
\begin{align*}
\Sinfo_k = r_{k + 1} - r_k \quad \mbox{and} \quad \Hblk_k =
\frac{1}{2} \int_{r_k}^{r_{k + 1}} \qdens(r) \, dr.
\end{align*}
For each block, the concavity of the square-root function
and Jensen's inequality ensure that
\begin{align}
  \label{EqnBlockUpper}
\int_{r_k}^{r_{k + 1}} \sqrt{\qdens(r)} \, dr & \leq \Sinfo_k
\sqrt{\frac{1}{\Sinfo_k} \int_{r_k}^{r_{k + 1}} \qdens(r) \, dr }
= \sqrt{2 \Sinfo_k \Hblk_k}.
\end{align}
Summing over the blocks yields the upper bound
\begin{align*}
  \int_0^{\log(\Tfinal/\Tinit)} \sqrt{\qdens(r)} \, dr
  & \leq \sum_{k=0}^{\Btot - 1} \sqrt{2 \Sinfo_k \Hblk_k}
  = \sqrt{2} \sum_{k=0}^{\Btot - 1} \sqrt{\Sinfo_k \Hblk_k},
\end{align*}
and squaring both sides, followed by division by $2$, yields the bound
(i) from equation~\eqref{EqnFinePartSandwich}.

To prove the fine-block limit~\eqref{EqnFineLimit}, let
$\{\Partition_n\}$ be a sequence of uniform refinements of the
log-time interval, let $\qdens_n$ be the piecewise-constant function
equal on each block to the average of $\qdens$ over that block, so
that $\|\qdens_n - \qdens\|_{L^1} \to 0$.  We have the upper bound
\begin{align*}
\left| \int_0^{\log(\Tfinal/\Tinit)} \sqrt{\qdens_n(r)} \, dr -
\int_0^{\log(\Tfinal/\Tinit)} \sqrt{\qdens(r)} dr \right | & \leq
\int_0^{\log(\Tfinal/\Tinit)} \left| \sqrt{\qdens_n(r)} -
\sqrt{\qdens(r)} \right| \, dr \\
&\leq \sqrt{ \log(\Tfinal/\Tinit) \|\qdens_n - \qdens\|_{L^1}} \to 0,
\end{align*}
where we used the inequality $|\sqrt{x} - \sqrt{y}|^2 \leq |x - y|$
combined with the Cauchy--Schwarz inequality.  Noting that $\sqrt{2
  \PartComp(\Partition_n)} = \int_0^{\log(\Tfinal/\Tinit)}
\sqrt{\qdens_n(r)} \, dr$ by the definition of partition complexity, it
follows that $\PartComp(\Partition_n)$ converges to the \dgc spread
$\PartHfine(\Tinit, \Tfinal)$, so the infimum over all partitions
equals the lower bound already established.
\end{proof}

\paragraph{Meaning of the \dgc spread:}
Using the shorthand notation $L \defn \log(\Tfinal/\Tinit)$, we have
the two expressions
\begin{align*}
  \PartComp([\Tinit, \Tfinal]) & = \frac{L}{2} \int_0^{L} \qdens(r) \,
  dr, \qquad \mbox{and} \qquad \PartHfine(\Tinit, \Tfinal) =
  \frac{1}{2} \left( \int_0^L \sqrt{\qdens(r)} \, dr \right)^2.
\end{align*}
Consequently, in the fine partition limit, the gains of multi-block
approaches are captured by the spread ratio
\begin{align}
\label{EqnRatio}  
\frac{L \int_0^{L} \qdens(r) \, dr}{ \left( \int_0^L \sqrt{\qdens(r)}
  \, dr \right)^2 } & \geq 1 \qquad \mbox{where $L \defn
  \log(\Tfinal/\Tinit)$.}
\end{align}
This ratio is equal to one if and only if $\qdens$ is constant almost
everywhere on $[0, L]$.  The slack in this inequality---and hence the
gains from the multi-block approach---become large when $\qdens$
spreads its mass in a highly non-uniform way across the interval.  It
is for this reason that we refer to it as the \dgc spread.

\subsubsection{Multi-block separation for Gaussian mixtures}
\label{SecSeparation}

The spread ratio~\eqref{EqnRatio} captures the gains achievable by a
multi-block approach, and it is natural to wonder how large it can be
made.  To give a sense of the gains that can be achieved, let us
consider a very simple example. In particular, suppose that a scalar
$\Zvar$ takes values on the two points $\{-R, R \}$ with equal probability, and
consider the forward heat path $X_t = \Zvar + B_t$ where $\{B_t \}_{t
  \geq 0}$ is a Brownian motion.  Suppose that our goal is to sample
from the $2$-component Gaussian mixture distribution $X_\Tinit = \Zvar
+ B_\Tinit$.  We fix the terminal time $\Tfinal = 2 R^2$, at which
$\Prob_\Tfinal$ is strongly log-concave with condition number at most
$2$.  The following result characterizes the single-block, $2$-block,
and \dgc spread complexities up to universal constants.  The bounds
all involve the ratio $R^2/\delta$, which is a natural complexity
measure, since it corresponds to the variance-normalized distance
between the two Gaussian components.

\mygraybox{
  \begin{lemma}[Multi-block gains for Gaussian mixture model]
    \label{LemTwoPointSep}
    For all $R^2/\Tinit$ sufficiently large, we have: \\
    \begin{center}
      \begin{tabular}{lcl}
        \bf{\mbox{Single block:}} & \hspace{0.2in} &
        $\PartComp([\Tinit, 2 R^2])  = \Theta(\log (R^2/\delta))$ \\[0.5em]
        \bf{\mbox{Two block:}} && $\PartComp(\Partition_2)  = \Theta(\log\log (R^2/\delta))$ \\[0.5em]
        \bf{\mbox{\dgc spread:}} && $\PartHfine(\Tinit, 2 R^2) =
        \Theta(1)$.
      \end{tabular}
    \end{center}
    Here the $2$-block result is achieved by the partition
    \begin{align}
      \label{Eqn2Partition}
      \Partition_2 \defn \{ [\delta, \tau_R], [\tau_R, 2 R^2] \} \qquad
      \mbox{where $\tau_R \defn \frac{R^2}{\log(R^2/\delta)}$.}
    \end{align}
  \end{lemma}
}
\noindent We prove this result in~\Cref{SecProofLemTwoPointSep}.  Note
that the single-block scaling and the \dgc spread scaling imply that
the spread ratio~\eqref{EqnRatio} grows as $\log (R^2/\Tinit)$.

Moreover, as the proof makes clear, it is possible to establish a more
general result for arbitrary $K$-block partitions.  For any $K = 2, 3,
\ldots$, one can construct a partition $\Partition_K$ such that
$\PartComp(\Partition_K) = \Theta \big( \log^{\otimes K}(R^2/\delta)
\big)$, where $\log^{\otimes K}$ denotes the composition of the
logarithm $K$ times.  As an example, for $K = 3$, we have
$\PartComp(\Partition_3) = \Theta(\log \log \log (R^2/\Tinit))$.  Note
that for this example, a \emph{very small} block number $\Btot$
suffices to render the dependence on $R^2/\delta$ essentially
negligible.

\subsubsection{Accuracy of uniform $\Btot$-block approximations}

In the previous argument, we made use of a carefully chosen
partition~\eqref{Eqn2Partition} with $\Btot = 2$ blocks; in practice,
finding such a partition would require running the dynamic program
described in~\Cref{SecTrackingBlock}.  As opposed to this type of
adaptively chosen partition, it is also interesting to study the
accuracy of the simplest type of $\Btot$-block approximation---namely,
one that makes use of blocks of constant size over the interval $[0,
  \log(\Tfinal/\Tinit)]$.  In terms of the shorthand $\Len \defn
\log(\Tfinal/\Tinit)$, this uniform choice leads to a partition $\UniK$
consisting of $\Btot$ blocks, each of length $\Len/\Btot$.
\Cref{FigKblock} gives plots of the $\UniK$-partition approximation
for $\Btot \in \{1, 2, 4 \}$ for two different underlying densities
$\qdens$: a $2$-component Gaussian mixture with relatively small
component variances in the top row; and the hierarchical mixture model
(see~\Cref{FigHier}) in the bottom row.
 \begin{figure}[h]
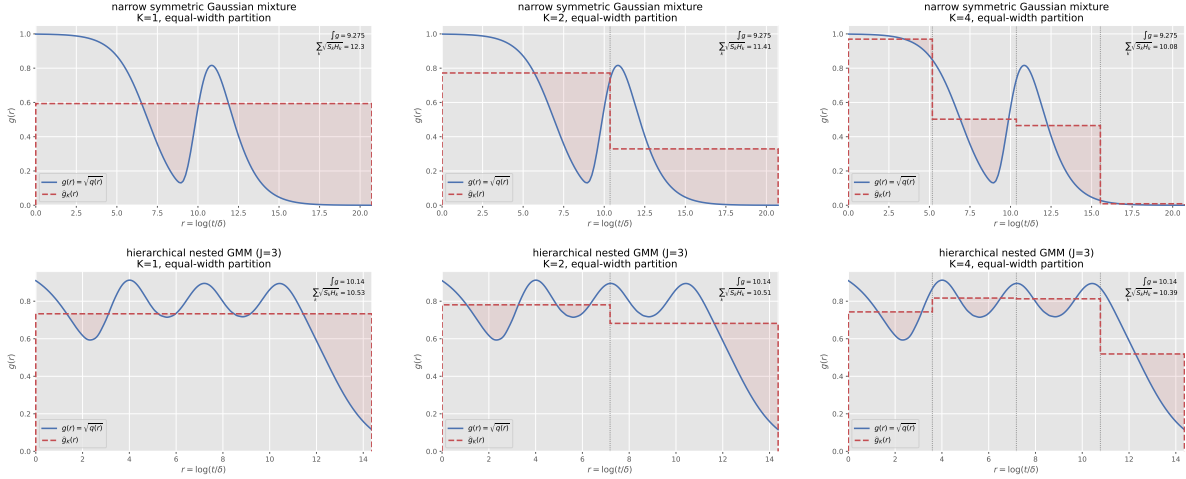

  \begin{center}
    \begin{tabular}{ccc}
      \widgraph{0.3\textwidth}{\figdir/fig_two_gaussian_k_01} &
      \widgraph{0.3\textwidth}{\figdir/fig_two_gaussian_k_02} &
      \widgraph{0.3\textwidth}{\figdir/fig_two_gaussian_k_04} \\
      \widgraph{0.3\textwidth}{\figdir/fig_hier_k_01} &
      \widgraph{0.3\textwidth}{\figdir/fig_hier_k_02} &
      \widgraph{0.3\textwidth}{\figdir/fig_hier_k_04}
    \end{tabular}    
    \caption{Plots of approximations of $\int_0^L \sqrt{\qdens(r)} dr$
      based on the uniform $K$-block partitions $\UniK$ for $K \in
      \{1, 2, 4 \}$.  The top row corresponds to a two-component
      Gaussian mixture with small variances, whereas the bottom row
      corresponds to a hierarchical mixture model.}
        \label{FigKblock}
  \end{center}
\end{figure}

 Let us now quantify how rapidly the block approximations based on
 $\UniK$ approach the fine-partition limit under some simple global
 conditions.  Suppose that $\fdens = \sqrt{\qdens}$ is continuous on
 $[0, \Len]$, and define its global modulus of continuity
\begin{align}
\label{EqnBlockAccuracyModulus}
  \omega_\fdens(h) & \defn \sup_{\substack{r, s \in [0, \Len]\\ |r -
      s| \leq h}} |\fdens(r) - \fdens(s)|.
\end{align}
As a particular case, when $\fdens$ is Lipschitz, then we have
$\omega_\fdens(h) \leq \Lip(\fdens) h$.
\mygraybox{
\begin{lemma}[Accuracy of $\Btot$-block approximations]
\label{LemBlockAccuracy}
For the uniform $\Btot$-block partition $\UniK$, we have
\begin{subequations}
\begin{align}
\label{EqnBlockAccuracyEqual}
\sqrt{2 \PartComp(\UniK)} & \leq \sqrt{2 \PartHfine(\Tinit, \Tfinal)}
+ \frac{\Len}{2} \omega_\fdens\left( \frac{\Len}{\Btot} \right).
\end{align}
If, in addition, the square-root density is lower bounded as
$\fdens(r) \geq m > 0$ on $[0, \Len]$, then
\begin{align}
\label{EqnBlockAccuracyIntegrated}
\sqrt{2 \PartComp(\UniK)} & \leq \sqrt{2 \PartHfine(\Tinit, \Tfinal)}
+ \frac{\Len}{2 m h^2} \int_0^h (h - u) \omega_\fdens(u)^2 \, du
\qquad \mbox{where $h \defn \frac{\Len}{\Btot}$.}
\end{align}
\end{subequations}
\end{lemma}
}
\noindent See~\Cref{SecProofLemBlockAccuracy} for the proof. \\

As an important special case, when $\fdens$ is Lipschitz with constant
$\Lip(\fdens)$, then these bounds imply that
\begin{subequations}
\begin{align}
\label{EqnBlockAccuracyLipschitz}
  \sqrt{2 \PartComp(\UniK)} \; \stackrel{(a)}{\leq} \; \sqrt{2
    \PartHfine(\Tinit, \Tfinal)} + \frac{\Lip(\fdens) \Len^2}{2
    \Btot}, \qquad \mbox{and} \quad \sqrt{2 \PartComp(\UniK)} \;
  \stackrel{(b)}{\leq} \; \sqrt{2 \PartHfine(\Tinit, \Tfinal)} +
  \frac{\Lip(\fdens)^2 \Len^3}{24 m \Btot^2}.
\end{align}
\end{subequations}
with inequality (b) holding when $\fdens(r) \geq m > 0$.  Observe that
this positivity condition improves the convergence rate from
$\Order(1/\Btot)$ to $\Order(1/\Btot^2)$.


\section{Proofs}
\label{SecProofs}

In this section, we collect the proofs of our main results, including
the proof of~\Cref{ThmMaster} in~\Cref{SecProofThmMaster}; and the
proof of~\Cref{ThmTrackMulti} in~\Cref{SecProofThmTrackMulti}.
Central to the proof of~\Cref{ThmMaster} is the one-step upper bound
on the KL error, stated as~\Cref{LemGeoOneStep} in~\Cref{LemOneStep}.
Proofs of all other claims, including all the model-specific results
stated in~\Cref{SecInfoUpper}, are deferred to the appendices.
  

\subsection{\texorpdfstring{Proof of~\Cref{ThmMaster}}%
  {Proof of the master theorem}}
\label{SecProofThmMaster}

We have stated our results using the standard heat path $X_t = \Zvar +
\sqrt{t} W$ parameterization of diffusion models.  For our analytical
purposes, it is more convenient to use the equivalent stochastic
localization (SL)~\cite{Eld13,Mon23} parameterization $\Ylam = \lam
X_{1/\lam}$ for $\lam > 0$. By construction, traversing the heat path
\emph{backwards} from $X_\Tfinal$ to $X_\Tinit$ is equivalent to
traversing the SL path \emph{forwards} from $Y_{\lamzero}$ to
$Y_{\lamone}$ where $\lamzero = 1/\Tfinal$ and $\lamone = 1/\Tinit$.
Since heat time represents variance, the inverse variable $\lambda$
represents the \emph{precision}.

\subsubsection{Converting to innovations space}

Define the denoiser $\Denoise_\lam(y) \defn \Exs[\Zvar \mid \Ylam =
  y]$.  By classical results in nonlinear filtering
theory~\cite{KalStr68,FujKalKun72}, the process $\{\Ylam \}_{\lam \geq
  0}$ can be represented as the solution of the \emph{innovations SDE}
\begin{subequations}
\begin{align}
  \label{EqnInnovationsSDE}
  d \Ylam & = \Denoise_\lam(\Ylam) \, d\lam + d \Blam, \qquad
  \mbox{where $\Blam$ is a Brownian motion.}
\end{align}
For a given set of grid points
$\lamzero = s_0 < s_1 < \cdots < s_N = \lamone$, the standard
Euler--Maruyama discretization of this stochastic innovations SDE is
given by
\begin{align}
  \label{EqnSIEuler}
  \YhatS{\jind + 1} & = \YhatS{\jind} + (s_{\jind + 1} - s_{\jind})
  \Denoise_{s_{\jind}}(\YhatS{\jind}) + \sqrt{s_{\jind + 1} -
    s_{\jind}} \, \xi_{\jind},
\end{align}
where $\xi_{\jind} \sim \Normal(0, \IdMat_\usedim)$.

This Euler scheme~\eqref{EqnSIEuler} has a one-to-one correspondence
with the modified heat-time Euler scheme~\eqref{EqnHeatAlgorithm}.  In
particular, let $\{\Yhat_\jind \}_{\jind \geq 0}$ be the SI-Euler
update~\eqref{EqnSIEuler} with grid points $\{s_\jind \}_{\jind \geq
  0}$, and let $\{\Xhat_\jind \}_{\jind \geq 0}$ be the heat-time
sequence~\eqref{EqnHeatAlgorithm}.  At each iteration $\jind$, the
correspondence is given by $ t_\jind = \frac{1}{s_\jind}$ and
$\Xhat_\jind = \frac{\Yhat_\jind}{s_\jind}$.  At the terminal
precision $s_N = \lamone = 1/\Tinit$, this correspondence gives
$\Xhat_N = \frac{\Yhat_N}{s_N} = \Tinit \Yhat_N$.  The precision-space
and heat-time denoisers are related by $\Denoise_\lam(y) =
\DenTime_{1/\lam}(y/\lam)$.  Finally, the precision-space MSE function
is $\gfun(\lam) \defn \Exs \big[ \| \Zvar - \Denoise_\lam(\Ylam)
  \|_2^2 \big]$, and the precision-space \dgc by
\begin{align}
  \label{EqnJinfoEquiv}
  \Jinfo(a, b) & \defn -\frac{1}{2}
  \int_a^b \lam \gfun'(\lam) \, d\lam,
  \quad \mbox{from which it follows that} \quad \Jinfo(a, b) =
  \Hinfo\left( \frac{1}{b}, \frac{1}{a} \right).
\end{align}
\end{subequations}

\subsubsection{Analysis of the one-step KL deficit} 
\label{LemOneStep}
For any pair $s, h > 0$, let $\Ptrue_{s, h}(y,\cdot)$ be the exact
transition kernel of the $\{\Ylam \}_{\lam > 0}$
process~\eqref{EqnInnovationsSDE} from $s$ to $s+h$.  The transition kernel defined by the discretized Euler
update~\eqref{EqnSIEuler} between $s$ and $s + h$ is
\begin{align}
 \label{EqnEulerOneStep}
\Qhat_{s, h}(u,\cdot) \defn \Normal \left(u + h \Denoise_s(u),h
\IdMat_\usedim\right),
\end{align}
where $\Normal(a,\mathbf{Q})$ denotes a Gaussian distribution with
mean $a$ and covariance $\mathbf{Q}$.  In particular, we establish the
following control:
\mygraybox{
\begin{lemma}[Multiplicative control of Euler one-step defect]
  \label{LemGeoOneStep}
For any $s, h > 0$, we have
\begin{align}
  \label{EqnGeoOneStep}
\underbrace{\Exs_{\Yvar_s} \left[\KL\left( \Ptrue_{s,
      h}(\Yvar_s,\cdot) \,\middle\|\, \Qhat_{s, h}(\Yvar_s,\cdot)
    \right) \right]}_{\equiv \EulKL(s, s + h) } & \leq \frac{h}{s}
\Jinfo(s, s + h).
\end{align}
\end{lemma}
}
\noindent See~\Cref{SecProofLemGeoOneStep} for the proof. \\

It is straightforward to complete the proof of~\Cref{ThmMaster} using
this result.  In particular, by combining the data processing
inequality with the chain rule for KL divergences, we have
\begin{align}
\KL(\Prob_\Tinit \| \Qprob_\Tinit) \; \leq \;
\sum_{\jind=0}^{\Nscore-1} \EulKL(s_\jind, s_{\jind+1}) +
\KL(\Prob_\Tfinal \| \Qprob_\Tfinal) & \; \stackrel{(i)}{\leq} \;
\sum_{\jind=0}^{\Nscore-1} \big \{ \frac{s_{\jind+1}}{s_\jind} - 1
\big \} \Jinfo(s_\jind, s_{\jind +1 }) + \KL(\Prob_\Tfinal \|
\Qprob_\Tfinal) \notag \\
\label{EqnFred}
& \stackrel{(ii)}{=} \sum_{\jind=0}^{\Nscore-1} \big \{
\frac{t_{\jind}}{t_{\jind+1}} - 1 \big \} \Hinfo(t_{\jind + 1},
t_{\jind} ) + \KL(\Prob_\Tfinal \| \Qprob_\Tfinal),
\end{align}
where step (i) follows by applying the bound~\eqref{EqnGeoOneStep}
with the choices $h = s_{\jind+1} - s_\jind$ and $s = s_\jind$ for
each block; and step (ii) follows by the stepsize conversion $t =
1/s$, and the relation $\Jinfo(a,b) = \Hinfo(1/b, 1/a)$ from
equation~\eqref{EqnJinfoEquiv}.

\paragraph{Extension to noisy scores:}  The extension to noisy
scores is straightforward.  Suppose that instead of using the true
denoising functions $\Denoise_s$, we make use of an estimate
$\widehat{\Denoise}_s$.  The Euler method then has the modified
one-step Gaussian transition kernel $\widehat{\Qhat}_{s, h}(u, \cdot)
\defn \Normal\left( u + h \widehat{\Denoise}_s(u), h \IdMat_\usedim
\right)$.  As shown in the proof of~\Cref{LemGeoOneStep}, the estimated
denoising functions lead to an additive perturbation, in that we have
the perturbed upper bound
\begin{align}
  \label{EqnGeoOneStepNoisy}
  \Exs_{\Yvar_s} \left[ \KL\left( \Ptrue_{s, h}(\Yvar_s, \cdot)
    \,\middle\|\, \widehat{\Qhat}_{s, h}(\Yvar_s, \cdot) \right)
    \right] & \leq \frac{h}{s} \Jinfo(s, s + h) + \frac{h}{2} \Exs
  \left[ \left\| \widehat{\Denoise}_s(\Yvar_s) - \Denoise_s(\Yvar_s)
    \right\|_2^2 \right].
\end{align}
Consequently, we can simply repeat the proof leading to the bound~\eqref{EqnFred},
and these score error terms will accumulate additively across blocks, leading
to the correction term~\eqref{EqnScoreError}.


\subsubsection{\texorpdfstring{Proof of~\Cref{LemGeoOneStep}}%
  {Proof of the geometric one-step lemma}}
\label{SecProofLemGeoOneStep}

Our proof is based on the exact representation
\begin{subequations}
\begin{align}
\label{EqnEulerExact}
\EulKL(s, s + h) & = \frac{h}{2} \gfun(s) - \frac{1}{2}
\int_{s}^{s + h} \gfun(r) \, dr,
\end{align}
which we prove below.  Taking this representation as given, let us
prove the claimed inequality.  Since we can write $\frac{h}{2}
\gfun(s) = \frac{1}{2} \int_{s}^{s + h} \gfun(s) \, dr$, we have
\begin{align}
  \label{eq:block-kl-one-interval}
  \EulKL(s, s + h) & = \frac{1}{2} \int_{s}^{s + h}
  \big\{ \gfun(s) - \gfun(r) \big\} \, dr.
\end{align}
\end{subequations}
Since the MSE function $\gfun$ is non-increasing, we have $\gfun'(u)
\leq 0$ pointwise, and hence
\begin{align*}
  \gfun(s) - \gfun(r) & = - \int_{s}^{r} \gfun'(u) \, du
  \leq - \int_{s}^{s + h} \gfun'(u) \, du,
\end{align*}
  using the fact that $\gfun'(u) \leq 0$ almost everywhere.  Combining
  this inequality with the
  representation~\eqref{eq:block-kl-one-interval} yields $\EulKL(s,
  s+h) \leq - \frac{h}{2} \int_{s}^{s + h} \gfun'(r) dr$. For any $r
  \geq s$, we have $h \{-\gfun'(r)\} = \frac{h}{r}\{-r \gfun'(r)\}
  \leq \frac{h}{s}\{-r\gfun'(r)\}$, and hence
  \begin{align*}
    \EulKL(s, s+h) & \leq \frac{h}{s} \; \left \{ -\frac{1}{2}
    \int_{s}^{s +h} r\gfun'(r) dr \right \} \; = \; \frac{h}{s} \;
    \Jinfo(s, s + h),
  \end{align*}
as claimed.

\paragraph{Proof of the exact representation~\eqref{EqnEulerExact}:}
Conditional on $\Yvar_s = u$, the exact transition has the
representation $\Yvar_{s+h} = u + h \Zvar+\sqrt h\,\xi$, where $\Zvar$
follows its posterior law given $\Yvar_s=u$ and the Gaussian noise
$\xi\sim\Normal(0,\IdMat_\usedim)$ is independent of $\Zvar$.  Taking
expectations with respect to $\Yvar_{s+h}$, we find that the entropy
of the exact transition is given by
\begin{align*}
-\Exs_{Y_{s+h}}[\log \Ptrue_{s,h}(u, Y_{s+h})] & = \frac{\usedim}{2}
\log(2\pi e h) + I(\Zvar; \Yvar_{s+h} \mid \Yvar_s = u),
\end{align*}
using the fact that conditional on both $\Zvar$ and $\Yvar_s=u$,
the transition noise is Gaussian with covariance $h \IdMat_\usedim$.
On the other hand, from the Gaussian form~\eqref{EqnEulerOneStep} of
the one-step Euler law, we can compute the cross-entropy
\begin{align*}
- \Exs_{Y_{s+h}}[\log \Qhat_{s,h}(u, Y_{s+h})] & = \frac{\usedim}{2}
\log(2\pi eh) + \frac{h}{2} \Tr \Cmat_s(u),
\end{align*}
where $\Cmat_s(u) \defn \Cov( \Zvar \mid \Yvar_s = u)$.

Using these two equations, we have
\begin{align*}
\KL\left( \Ptrue_{s, h}(u,\cdot) \,\middle\|\, \Qhat_{s, h}(u,\cdot)
\right) = \Exs_{Y_{s+h}} \Big[ \log \frac{\Ptrue_{s,h}(u,
    Y_{s+h})}{\Qhat_{s,h}(u; Y_{s+h})} \Big] & = \frac{h}{2} \Tr
\Cmat_s(u) - \ShannonInfo(\Zvar;\Yvar_{s+h}\mid\Yvar_s = u).
\end{align*}
Taking expectations over the marginal distribution of $\Yvar_s$ yields
\begin{align}
\label{EqnMina}  
\EulKL(s, s + h) & = \frac{h}{2} \gfun(s)
- \ShannonInfo(\Zvar; \Yvar_{s + h} \mid \Yvar_s).
\end{align}
Finally, by the conditional I--MMSE identity for the Gaussian
observation process~\cite{GuoEtAl05}, for almost every $r > s$, we
have $\frac{1}{2} \gfun(r) = \frac{d}{d r} \ShannonInfo(\Zvar; \Yvar_r
\mid \Yvar_s)$.  Integrating from $s$ to $s + h$ yields
\begin{align*}
\ShannonInfo(\Zvar; \Yvar_{s + h} \mid \Yvar_s) & =
\underbrace{\ShannonInfo(\Zvar; \Yvar_s \mid \Yvar_s)}_{= 0} \; + \;
\frac{1}{2} \int_s^{s + h} \gfun(r) \, dr.
\end{align*}
Combining with equation~\eqref{EqnMina} completes the proof.

\paragraph{Extension to learned scores:}
For the learned kernel, expanding the Gaussian cross-entropy after
shifting the Euler mean from $u + h \Denoise_s(u)$ to $u + h
\widehat{\Denoise}_s(u)$ shows that the cross term vanishes exactly,
since $\Exs\left[ \Yvar_{s + h} - \Yvar_s - h \Denoise_s(\Yvar_s)
  \,\middle|\, \Yvar_s \right] = 0$.  Consequently, the increase in
expected one-step KL is given by the additive term $\frac{h}{2} \Exs
\left[ \left\| \widehat{\Denoise}_s(\Yvar_s) - \Denoise_s(\Yvar_s)
  \right\|_2^2 \right]$, as claimed.



\subsection{Proof of~\Cref{ThmTrackMulti}}
\label{SecProofThmTrackMulti}

We split our proof into two parts, corresponding to the upper
bound~\eqref{EqnFixedBudgetKL} and the lower bound~\eqref{EqnDGCLower}.

\subsubsection{Proof of the upper bound~\eqref{EqnFixedBudgetKL}}

It is convenient to introduce the shorthand $a_\bind \defn 4
\frac{\sqrt{\PartH}}{\Nscore}
\sqrt{\frac{\Sinfo_\bind}{\Hblk_\bind}}$, so that $\rho_\bind =
\min\{1, a_\bind\}$.

\paragraph{Certifying the accuracy:}  We first prove that
$\sum_{\bind =0}^{\Btot-1} \rho_\bind \Hinfo_k \leq \frac{4
  \PartH}{\Nscore}$.  Since $\rho_\bind \leq a_\bind$, we have
\begin{align*}
  \sum_{\bind = 0}^{\Btot - 1} \rho_\bind \Hblk_\bind & \leq 4
  \frac{\sqrt{\PartH}}{\Nscore} \sum_{\bind = 0}^{\Btot - 1}
  \sqrt{\Sinfo_\bind \Hblk_\bind} = \frac{4 \PartH}{\Nscore},
\end{align*}
where we have used the definition of $\PartH$.

\paragraph{Certifying the iteration count:}  Next we show
that our choice of multipliers $\rho_\bind$ leads to block iteration
counts that satisfy the constraint $\sum_{\bind=0}^{\Btot-1} N_\bind \leq \Nscore$.
Traversing block $\bind$ requires a total of $N_\bind = \left\lceil
\Sinfo_\bind/ \log(1 + \rho_\bind) \right\rceil$ score evaluations.
Thus, we have
\begin{align*}
\sum_{\bind = 0}^{\Btot - 1} N_\bind \leq \Btot + \sum_{\bind =
  0}^{\Btot - 1} \frac{\Sinfo_\bind}{\log(1 + \rho_\bind)} \;
\stackrel{(i)}{\leq} \; \Btot + 2 \sum_{\bind = 0}^{\Btot - 1}
\frac{\Sinfo_\bind}{\rho_\bind} & \stackrel{(ii)}{\leq} \Btot + 2
\sum_{\bind = 0}^{\Btot - 1} \Big \{ \Sinfo_\bind +
\frac{\Sinfo_\bind}{a_\bind} \Big \} \\
& \stackrel{(iii)}{\leq} \Btot + 2 \log(\Tfinal/\Tinit) +
\frac{\Nscore}{4} \\
& \stackrel{(iv)}{\leq}  \Nscore,
\end{align*}
where step (i) follows since $\log(1 + \rho) \geq \rho/2$ for $\rho
\in [0,1]$; step (ii) follows since $\frac{1}{\rho_\bind} =
\max\left\{ 1, \frac{1}{a_\bind} \right\} \leq 1 + \frac{1}{a_\bind}$;
and step (iii) follows since $\sum_{\bind=0}^{\Btot - 1} \Sinfo_\bind
= \log(\Tfinal/\Tinit)$, and
\begin{align*}
  \sum_{\bind =0}^{\Btot-1} \frac{\Sinfo_\bind}{a_\bind} = \frac{1}{4}
  \sum_{\bind=0}^{\Btot-1} \Sinfo_\bind \frac{\Nscore}{\sqrt{\PartH}}
  \sqrt{\frac{\Hblk_\bind}{\Sinfo_\bind}} \; = \; \frac{N}{2}
  \frac{1}{\sqrt{\PartH}} \sum_{\bind=0}^{\Btot-1} \sqrt{\Sinfo_\bind
    \Hblk_\bind} \; = \; \frac{\Nscore}{2}
\end{align*}
using the definition of $\PartH$. Finally, step (iv) follows from the
assumed lower bound on the iteration number $\Nscore \geq 2 \big(
\Btot + 2 \log(\Tfinal/\Tinit) \big)$.

\subsubsection{Proof of the lower bound~\eqref{EqnDGCLower}}
Fix an arbitrary feasible integer allocation $(N_0, \ldots, N_{\Btot -
  1})$ in equation~\eqref{EqnKblockDP}, so that we have $N_\bind \geq
1$ and $\sum_{\bind = 0}^{\Btot - 1} N_\bind = N$.  Using these block
iteration counts, we define the geometric multiplier
$\rho_\bind^{\mathrm{alloc}} \defn \exp\left\{
\frac{\Sinfo_\bind}{N_\bind} \right\} - 1$.  Since $e^x - 1 \geq x$,
we have the lower bound
\begin{align}
\label{EqnDumpling}  
  \sum_{\bind = 0}^{\Btot - 1} \rho_\bind^{\mathrm{alloc}} \Hblk_\bind
  & \geq \sum_{\bind = 0}^{\Btot - 1} \frac{\Sinfo_\bind
    \Hblk_\bind}{N_\bind}.
\end{align}
Consequently, we have the upper bound
\begin{align*}
\PartH = \left( \sum_{\bind = 0}^{\Btot - 1} \sqrt{\Sinfo_\bind
  \Hblk_\bind} \right)^2 & \stackrel{(i)}{\leq} \left( \sum_{\bind =
  0}^{\Btot - 1} \frac{\Sinfo_\bind \Hblk_\bind}{N_\bind} \right)
\underbrace{\left( \sum_{\bind = 0}^{\Btot - 1} N_\bind \right)}_{= N}
\; \stackrel{(ii)}{\leq} \; \left (\sum_{\bind = 0}^{\Btot - 1}
\rho_\bind^{\mathrm{alloc}} \Hblk_\bind \right) \; N,
\end{align*}
where step (i) follows from the Cauchy--Schwarz inequality; and step
(ii) follows from inequality~\eqref{EqnDumpling}.  Re-arranging yields
the lower bound~\eqref{EqnDGCLower}.


\subsubsection*{Acknowledgements} This work was partially supported by
a Guggenheim Fellowship, an NSF grant (DMS-2311072), and the Ford
Professorship at MIT.


\bibliographystyle{alpha_initials}

{\small{
\bibliography{merged_clean}
}}


\appendix



\section{Properties of the \dgc function}
\label{AppPath}

In this section, we collect together the proofs of various properties
of the \dgc function~\eqref{EqnDefnHinfo}.

\subsection{Proof of the information representation~\eqref{EqnHinfoInfoRep}}
\label{SecProofEqnHinfoRep}

By the chain rule for derivatives, we have the identity
$\frac{\hfun'(t)}{t} = \frac{d}{dt} \left\{ \frac{\hfun(t)}{t}
\right\} + \frac{\hfun(t)}{t^2}$.  Consequently, we can write
\begin{align*}
  \Hinfo(a, b) & \stackrel{(i)}{=} \frac{1}{2} \int_a^b
  \frac{\hfun'(t)}{t} dt \; \stackrel{(ii)}{=}\; \frac{1}{2} \left\{
  \frac{\hfun(b)}{b} - \frac{\hfun(a)}{a} \right\} + \frac{1}{2}
  \int_a^b \frac{\hfun(t)}{t^2} \, dt,
\end{align*}
where step (i) follows from the definition~\eqref{EqnDefnHinfo} of
$\Hinfo$, and step (ii) follows by substituting our chain rule
expression and computing the integral for the first term.  For the heat
channel $X_t = \Zvar + \sqrt{t} W$, the I--MMSE
identity~\cite{GuoEtAl05}
asserts that $\frac{d}{dt} \Info(\Zvar; X_t) = -\frac{\hfun(t)}{2
  t^2}$.  Integrating from $a$ to $b$ and substituting into the
preceding display yields
\begin{align*}
  \Hinfo(a, b) & = \Info(\Zvar; X_a) - \Info(\Zvar; X_b) + \frac{1}{2}
  \left\{ \frac{\hfun(b)}{b} - \frac{\hfun(a)}{a} \right\},
\end{align*}
as claimed.


\subsection{Proof of~\Cref{LemDinfoSandwich}}
\label{SecProofLemDinfoSandwich}

Recall the forward heat path $X_t = \Zvar + B_t$, where $\{B_t \}_{t
  \geq 0}$ is a Brownian motion.  We observe that the denoising
operator $\DenTime_t$ satisfies the orthogonality property $\Exs
\big[\inprod{\Zvar - \DenTime_t(X_t)}{\DenTime_t(X_t)} \big] = 0$,
from which it follows that $\hfun(t) = \Exs \|\Zvar - \DenTime_t(X_t)\|_2^2
= \Exs \| \Zvar \|_2^2 - \Exs \|\DenTime_t(X_t)\|_2^2$.
Consequently, we have the equivalence
\begin{align}
\label{EqnHfunEquiv}  
\hfun(t) - \hfun(s) = \Exs \|\DenTime_s(X_s)\|_2^2 - \Exs
\|\DenTime_t(X_t)\|_2^2,
\end{align}
which we use in the proof below.

Beginning with the definition~\eqref{EqnDefnDinfo} of $\Dinfo$,
expanding the squared norm yields
\begin{align*}
\Dinfo(s, t) = \Exs \| \DenTime_s(X_s) - \DenTime_t(X_t) \|_2^2 & =
\Exs \| \DenTime_s(X_s) \|_2^2 - 2 \Exs \big[
  \inprod{\DenTime_s(X_s)}{\DenTime_t(X_t)} \big] + \Exs \|
\DenTime_t(X_t) \|_2^2.
\end{align*}
From equation~\eqref{EqnReverseMartingale}, recall the reverse
martingale property $\Exs [\DenTime_s(X_s) \mid X_t ] =
\DenTime_t(X_t)$ for pairs $0 < s < t$.  Using this fact combined with
the tower relation for conditional expectations, we can compute
\begin{align*}
\Exs \big[ \inprod{\DenTime_s(X_s)}{\DenTime_t(X_t)} \big] & = \Exs
\big[ \inprod{\Exs[\DenTime_s(X_s) \mid X_t]}{\DenTime_t(X_t)} \big]
  \; = \; \Exs \| \DenTime_t(X_t) \|_2^2,
\end{align*}
and as a consequence $\Dinfo(s, t) = \Exs \| \DenTime_s(X_s) \|_2^2 -
\Exs \| \DenTime_t(X_t) \|_2^2$.  Combining with the
equivalence~\eqref{EqnHfunEquiv} proves the
claim~\eqref{EqnDinfoExplicit}.


\section{Results on data-dependent estimation}

In this section, we collect together the proofs of various results in
data-dependent estimation, including the proof of~\Cref{PropDataSingle}
in~\Cref{SecProofPropDataSingle} and the proof of the perturbed
sandwich~\eqref{EqnRobustSandwich}
in~\Cref{SecProofEqnRobustSandwich}.


\subsection{Proof of~\Cref{PropDataSingle}}
\label{SecProofPropDataSingle}

We claim that it suffices to establish the two-sided tail bound
\begin{align}
  \label{EqnPogacar}
  \left| \Hupper(a,b) - \Hhat(a,b) \right| & \leq \HackErr \qquad
  \mbox{with probability at least \(1 - \eta\),}
\end{align}
with the residual $\HackErr$ from equation~\eqref{EqnDefnHackErr}.
Indeed, on the event that the bound~\eqref{EqnPogacar} holds,
the sandwich
relation~\eqref{EqnNewSandwich} between the pair $\Hinfo(a,b)$ and
$\Hupper(a,b)$ gives
\begin{align*}
  \Hinfo(a,b) & \leq \Hupper(a,b) \leq \Hhat(a,b) + \HackErr, \quad
  \mbox{as well as} \\
  \Hhat(a,b) + \HackErr & \leq \Hupper(a,b) + 2 \HackErr \leq 2 \big\{
  \Hinfo(a,b) + \HackErr \big\}.
\end{align*}
Combining these two inequalities yields the sandwich
claim~\eqref{EqnDataSingle} of~\Cref{PropDataSingle}.

\paragraph{Proof of the tail bound~\eqref{EqnPogacar}:} Introduce
the shorthand $\Delta_\ell \defn
\DenTime_{\dyadic_\ell}(X_{\dyadic_\ell}) - \DenTime_{\dyadic_{\ell +
    1}}(X_{\dyadic_{\ell + 1}})$ and recall that $\Qfun = \frac{1}{2}
\sum_{\ell = 0}^{\DyaTot - 1}
\frac{\|\Delta_\ell\|_2^2}{\dyadic_\ell}$.  By Minkowski's inequality
in $L^{p/2}$, together with the triangle inequality in $L^p$, we have
\begin{align*}
  \left( \Exs\|\Delta_\ell\|_2^p \right)^{1/p} \leq \left( \Exs\left\|
  \DenTime_{\dyadic_\ell}(X_{\dyadic_\ell}) \right\|_2^p \right)^{1/p}
  + \left( \Exs\left\| \DenTime_{\dyadic_{\ell + 1}}(X_{\dyadic_{\ell
      + 1}}) \right\|_2^p \right)^{1/p} \leq 2 M_p,
\end{align*}
and consequently
\begin{align*}
  \left( \Exs[\Qfun^{p/2}] \right)^{2/p} & \leq \frac{1}{2} \sum_{\ell
    = 0}^{\DyaTot - 1} \frac{ \left( \Exs\|\Delta_\ell\|_2^p
    \right)^{2/p} }{\dyadic_\ell} \; \leq \; 2 M_p^2 \sum_{\ell =
    0}^{\DyaTot - 1} \frac{1}{\dyadic_\ell} \leq \frac{4 M_p^2}{a},
\end{align*}
where we used the dyadic relation $\dyadic_\ell = 2^\ell a$.

We now bound the bias in the truncated variable $\Qfun_\tau^{(i)}
\defn \min\{ \Qfunup{i}, \tau \}$.  Since $p/2 > 1$, we have
\begin{align}
  \label{EqnTruncatedBias}
  \Exs[\Qfun] - \Exs[\Qfun_\tau^{(i)}] & = \Exs[(\Qfun - \tau)_+] \leq
  \frac{\Exs[\Qfun^{p/2}]}{\tau^{p/2 - 1}} \leq \frac{ \left( 4 M_p^2
    / a \right)^{p/2} }{ \tau^{p/2 - 1} }.
\end{align}

Next we apply the empirical Bernstein inequality of Maurer and
Pontil~\cite[Theorem~4]{MauPont09}.  Apply this result to the
independent variables $\Qfun_\tau^{(i)} / \tau \in [0,1]$, and then to
$1 - (\Qfun_\tau^{(i)} / \tau)$, in each case with failure probability
$\eta / 2$.  Since the two collections have the same sample variance,
a union bound over the pair yields
\begin{align*}
  \left| \Exs[\Qfun_\tau^{(i)}] - \Hhat(a,b) \right| & \leq \sqrt{
    \frac{ 2 \Vhat \log(4 / \eta) }{\numobs} } + \frac{ 7 \tau \log(4
    / \eta) }{ 3(\numobs - 1) }
\end{align*}
with probability at least $1 - \eta$.  Recalling that $\Exs[\Qfun] =
\Hupper(a,b)$, and combining this inequality with the bias
bound~\eqref{EqnTruncatedBias} gives us
\begin{align*}
  \left| \Hupper(a,b) - \Hhat(a,b) \right| & \leq \sqrt{ \frac{ 2
      \Vhat \log(4 / \eta) }{\numobs} } + \frac{ 7 \tau \log(4 / \eta)
  }{ 3(\numobs - 1) } + \frac{ \left( 4 M_p^2 / a \right)^{p/2} }{
    \tau^{p/2 - 1} }.
\end{align*}
This bound holds for every deterministic choice of the truncation
level \(\tau\).  Substituting the choice of \(\tau\) from
equation~\eqref{EqnDefnHhat}, and using $\sup_{p \geq 4} \frac{p}{p -
  2} \left( \frac{p - 2}{2} \right)^{2/p} \leq 2$ yields the
claim~\eqref{EqnPogacar}.

  
\subsection{Proof of the perturbed sandwich relation}
\label{SecProofEqnRobustSandwich}

Define the score error increments $e_\ell \defn
\DenTimeHat_{\dyadic_\ell}(X_{\dyadic_\ell}) -
\DenTime_{\dyadic_\ell}(X_{\dyadic_\ell})$ for \mbox{$\ell = 0,
  \ldots, \DyaTot$,} and for $\ell = 0, \ldots, \DyaTot - 1$, the
exact and learned increments
\begin{align*}
  U_\ell & \defn \DenTime_{\dyadic_\ell}(X_{\dyadic_\ell}) -
  \DenTime_{\dyadic_{\ell + 1}}(X_{\dyadic_{\ell + 1}}), \quad
  \mbox{and} \quad \widehat U_\ell \defn
  \DenTimeHat_{\dyadic_\ell}(X_{\dyadic_\ell}) -
  \DenTimeHat_{\dyadic_{\ell + 1}}(X_{\dyadic_{\ell + 1}}).
\end{align*}
Observe that these definitions ensure that $\widehat U_\ell = U_\ell +
e_\ell - e_{\ell + 1}$.  In terms of the weighted quadratic norm
$\|V\|_{\mathcal H}^2 \defn \frac{1}{2} \sum_{\ell = 0}^{\DyaTot - 1}
\frac{\Exs \|V_\ell\|_2^2}{\dyadic_\ell}$, we have the convenient
equivalences
\begin{align*}
\Hupper(a,b) = \|U\|_{\mathcal H}^2, \quad \mbox{and} \quad \Hmod(a,b)
= \|\widehat U\|_{\mathcal H}^2.
\end{align*}
Hence, the reverse triangle inequality yields
\begin{align*}
  \left| \sqrt{\Hmod(a, b)} - \sqrt{\Hupper(a, b)} \right| & \leq
  \left\| \{e_\ell - e_{\ell + 1}\}_{\ell = 0}^{\DyaTot - 1}
  \right\|_{\mathcal H}.
\end{align*}
Using the upper bound $\|e_\ell - e_{\ell + 1}\|_2^2 \leq 2
\|e_\ell\|_2^2 + 2 \|e_{\ell + 1}\|_2^2$, the dyadic relation
$\dyadic_{\ell + 1} \leq 2 \dyadic_\ell$, and shifting the index in
the second sum, we obtain
\begin{align*}
  \left\| \{e_\ell - e_{\ell + 1}\}_{\ell = 0}^{\DyaTot - 1}
  \right\|_{\mathcal H}^2 & \leq \sum_{\ell = 0}^{\DyaTot - 1}
  \frac{\Exs \|e_\ell\|_2^2}{\dyadic_\ell} + 2 \sum_{\ell =
    1}^{\DyaTot} \frac{\Exs \|e_\ell\|_2^2}{\dyadic_\ell} \leq 3
  \DenError(a, b).
\end{align*}
This bound implies that
\begin{align}
\label{EqnHuez}  
  \left( \sqrt{\Hmod(a, b)} - \sqrt{3 \DenError(a, b)} \right)_+ \;
  \leq \; \sqrt{\Hupper(a, b)} \; \leq \; \sqrt{\Hmod(a, b)} + \sqrt{3
    \DenError(a, b)}.
\end{align}
Recall that $\frac{1}{2} \Hupper(a, b) \leq \Hinfo(a, b) \leq
\Hupper(a, b)$ from the sandwich guarantee~\eqref{EqnNewSandwich}.
Squaring the bounds~\eqref{EqnHuez} and combining with this sandwich
yields
\begin{align*}
\frac{1}{2} \left( \sqrt{\Hmod(a, b)} - \sqrt{3 \DenError(a, b)}
\right)_+^2 \; \leq \; \Hinfo(a, b) \; \leq \; \left( \sqrt{\Hmod(a,
  b)} + \sqrt{3 \DenError(a, b)} \right)^2,
\end{align*}
as claimed in equation~\eqref{EqnRobustSandwich}.


\section{Proofs of model-specific results}

In this section, we collect together the proofs of the propositions
from~\Cref{SecInfoUpper}.


\subsection{Proof of~\Cref{PropCovariance}}
\label{SecProofPropCovariance}

Let $\Gvar_1$ and $\Gvar_2$ be independent standard Gaussian vectors,
independent of $\Zvar$, and consider the coupled heat-path
observations $\Xvar_\Tinit \defn \Zvar + \sqrt{\Tinit} \Gvar_1$, and
$\Xvar_\Tfinal \defn \Xvar_\Tinit + \sqrt{\Tfinal - \Tinit} \Gvar_2$.
This coupling produces the Markov chain $\Zvar \longrightarrow
\Xvar_\Tinit \longrightarrow \Xvar_\Tfinal$, so that by the chain rule
for mutual information, we can write
\begin{subequations}
  \begin{align}
    \label{EqnConditionalGaussianChainRule}
    \ShanInfo(\Zvar; \Xvar_\Tinit) - \ShanInfo(\Zvar; \Xvar_\Tfinal) &
    = \ShanInfo(\Zvar; \Xvar_\Tinit \mid \Xvar_\Tfinal) \; = \;
    \Ent(\Xvar_\Tinit \mid \Xvar_\Tfinal) - \Ent(\Xvar_\Tinit \mid
    \Zvar, \Xvar_\Tfinal).
  \end{align}
  Conditional on $\Zvar$, the two noise vectors $\Xvar_\Tinit - \Zvar$
  and $\Xvar_\Tfinal - \Zvar$ are jointly Gaussian, with respective
  covariance matrices $\Tinit \IdMat_\usedim$ and $\Tfinal
  \IdMat_\usedim$ and cross-covariance $\Tinit \IdMat_\usedim$.  By
  the formula for Gaussian conditional covariances, we have
  $\Cov(\Xvar_\Tinit \mid \Zvar, \Xvar_\Tfinal) = \Tinit
  \IdMat_\usedim - \Tinit \IdMat_\usedim (\Tfinal \IdMat_\usedim)^{-1}
  \Tinit \IdMat_\usedim = \frac{\Tinit(\Tfinal - \Tinit)}{\Tfinal}
  \IdMat_\usedim$.  By the formula for multivariate Gaussian entropy,
  we have
  \begin{align}
    \label{EqnConditionalGaussianBridgeEntropy}
    \Ent(\Xvar_\Tinit \mid \Zvar, \Xvar_\Tfinal)
    & = \frac{\usedim}{2}
    \log\left(
      2 \pi e \frac{\Tinit(\Tfinal - \Tinit)}{\Tfinal}
    \right).
  \end{align}

Translating $\Zvar$, $\Xvar_\Tinit$, and $\Xvar_\Tfinal$ does not
change any of the entropies or mutual informations above, so we may
assume without loss of generality that $\Exs \Zvar = 0$.  Define the
matrix $\Amat \defn (\CovZ + \Tinit \IdMat_\usedim) (\CovZ + \Tfinal
\IdMat_\usedim)^{-1}$.  Translation invariance of conditional entropy
and the fact that conditioning reduces entropy yields
\begin{align*}
  \Ent(\Xvar_\Tinit \mid \Xvar_\Tfinal)
  & = \Ent(\Xvar_\Tinit - \Amat \Xvar_\Tfinal
  \mid \Xvar_\Tfinal) \; \leq \; \Ent(\Xvar_\Tinit - \Amat \Xvar_\Tfinal).
\end{align*}
Under the Markov coupling, we can compute the covariance matrices
$\Cov(\Xvar_\Tinit) = \CovZ + \Tinit \IdMat_\usedim$,
\mbox{$\Cov(\Xvar_\Tfinal) = \CovZ + \Tfinal \IdMat_\usedim$,} and
$\Cov(\Xvar_\Tinit, \Xvar_\Tfinal) = \CovZ + \Tinit \IdMat_\usedim$.
It follows that the residual covariance is
\begin{align}
    \Cov(\Xvar_\Tinit - \Amat \Xvar_\Tfinal) & = \CovZ + \Tinit
    \IdMat_\usedim - (\CovZ + \Tinit \IdMat_\usedim) (\CovZ + \Tfinal
    \IdMat_\usedim)^{-1} (\CovZ + \Tinit \IdMat_\usedim) \notag\\
    \label{EqnConditionalGaussianResidualCovariance}    
    & = (\Tfinal - \Tinit) (\CovZ + \Tinit \IdMat_\usedim) (\CovZ +
    \Tfinal \IdMat_\usedim)^{-1}.
  \end{align}
Applying the Gaussian maximum-entropy inequality to the residual in
equation~\eqref{EqnConditionalGaussianResidualCovariance} gives
\begin{align*}
  \Ent(\Xvar_\Tinit \mid \Xvar_\Tfinal) & \leq
\frac{1}{2} \log \det\left[ 2 \pi e (\Tfinal - \Tinit) (\CovZ + \Tinit
  \IdMat_\usedim) (\CovZ + \Tfinal \IdMat_\usedim)^{-1} \right].
  \end{align*}
Subtracting equation~\eqref{EqnConditionalGaussianBridgeEntropy} from
equation~\eqref{EqnConditionalGaussianChainRule} and collecting scalar
determinant factors yields
  \begin{align*}
\ShanInfo(\Zvar; \Xvar_\Tinit) - \ShanInfo(\Zvar; \Xvar_\Tfinal) &
    \leq \frac{1}{2} \log \det\left[ \frac{\Tfinal}{\Tinit} (\CovZ +
      \Tinit \IdMat_\usedim) (\CovZ + \Tfinal \IdMat_\usedim)^{-1}
      \right] \; = \; \frac{1}{2} \log \det\left[ \left(
      \IdMat_\usedim + \frac{\CovZ}{\Tinit} \right) \left(
      \IdMat_\usedim + \frac{\CovZ}{\Tfinal} \right)^{-1} \right].
  \end{align*}
Diagonalizing $\CovZ$ to expose the eigenvalues $\{\eigval_j
\}_{j=1}^d$ yields the claimed bound.
\end{subequations}


\subsection{Proof of~\Cref{PropGoodRateDistor}}
\label{SecProofPropGoodRateDistor}

For any encoding $\Prob_{\Zhat \mid \Zvar, X_T}$, the chain rule for
mutual information gives
\begin{align}
  \label{EqnLao}
  \ShanInfo(\Zvar; X_\Tinit \mid X_T) & \leq \ShanInfo(\Zvar; \Zhat
  \mid X_T) + \ShanInfo(\Zvar; X_\Tinit \mid \Zhat, X_T).
\end{align}
Now let $\{\Bvar_\lambda\}_{\lambda \geq 0}$ be a standard
$\usedim$-dimensional Brownian motion independent of $\Zvar$, and
consider the process $\Yvar_\lambda = \lambda \Zvar + \Bvar_\lambda$.
We form the coupling $X_T = T \Yvar_{1/T}$ and $X_\Tinit = \Tinit
\Yvar_{1/\Tinit}$, and generate $\Zhat$ from the conditional
distribution $\Prob_{\Zhat \mid \Zvar, X_T}$ using auxiliary
randomness independent of the future Brownian increments.  Conditional
on $X_T$, the observations $\Ylam$ for $\lam \in [1/\Tfinal,
  1/\Tinit]$ correspond to additional independent Gaussian
observations of $\Zvar$.  Consequently, we can write
\begin{align*}
  \ShanInfo(\Zvar; X_\Tinit \mid \Zhat, X_T) \; \stackrel{(i)}{=} \;
  \frac{1}{2} \int_{1/T}^{1/\Tinit} \Exs \left\| \Zvar - \Exs[\Zvar
    \mid \Zhat, X_T, \Yvar_\lambda] \right\|_2^2 \, d\lambda \; &
  \stackrel{(ii)}{\leq} \; \frac{1}{2} \int_{1/T}^{1/\Tinit} \Exs
  \| \Zvar - \Zhat \|_2^2 \, d\lambda \\
  & = \frac{1}{2} \big \{ \frac{1}{\delta} - \frac{1}{T} \big \}
  \Exs \| \Zvar - \Zhat \|_2^2,
\end{align*}
where step (i) follows from the conditional I--MMSE
identity~\cite{GuoEtAl05}; and step (ii) follows since $\Zhat$ is an
admissible estimator of $\Zvar$, and the conditional expectation is
the optimal estimator.  Combining this bound with our initial
decomposition~\eqref{EqnLao} yields the claim.


\subsection{\texorpdfstring{Proof of~\Cref{PropPoincare}}%
  {Proof of the Poincare proposition}}
\label{SecProofPropPoincare}

Letting $g_u(z) = \inprod{u}{z}$ be a linear function, the \Poincare
condition~\eqref{EqnDefnPoincare} implies that
\begin{align*}
  u^T \CovZ u \; = \; \var(g_u(\Zvar)) & \leq \Pcar \Exs \|\nabla
  g_u(Z)\|_2^2 = \Pcar \|u\|_2^2,
\end{align*}
where we have used the fact that $\nabla g_u(z) = u$.  Since this
bound holds for vectors $u \in \real^d$, it follows that $\CovZ
\preceq \Pcar \IdMat$.

Since $\Exs Z = 0$, for the Gaussian initialization $\Qprob_\Tfinal
\sim \Normal(0, \Tfinal \IdMat)$, we have $\KL(\Prob_\Tfinal \|
\Qprob_\Tfinal) \leq \frac{\Tr \CovZ}{2 \Tfinal} \leq
\frac{\varepsilon}{4}$.  Moreover, the \Poincare inequality gives
$\frac{\hfun(\Tfinal)}{\Tfinal} \leq \frac{\Pcar \usedim}{\Tfinal}
\leq \frac{\varepsilon}{2}$.  Consequently,
applying~\Cref{PropCovariance} yields the upper bound $\Hinfo(\Tinit,
\Tfinal) \leq \frac{\usedim}{2} \log\left(1 +
\frac{\Pcar}{\Tinit}\right) + \frac{\varepsilon}{2}$.  We now choose
$N \defn \left\lceil \frac{4}{\varepsilon} \left\{ \usedim \log\left(1
+ \frac{\Pcar}{\Tinit}\right) + 1 \right\}
\log\left(\frac{\Tfinal}{\Tinit}\right) \right\rceil$.  Since
$\varepsilon \in (0,1)$, the single-block discretization bound gives
$\frac{2 \Hinfo(\Tinit, \Tfinal) \log(\Tfinal/\Tinit)}{N} \leq
\frac{\varepsilon}{4}$.  Combining this with the initialization bound,
we find that $\KL(\Prob_\Tinit \| \Qprob) \leq \frac{\varepsilon}{4} +
\frac{\varepsilon}{4} = \frac{\varepsilon}{2}$.  Thus, the
single-block procedure with this choice of $N$ samples from
$\Prob_\Tinit$ with iteration complexity
\begin{align}
 \label{EqnPoincareInter}
\NscoreTot(\varepsilon) & \leq 1 + \frac{4}{\varepsilon} \; \Big \{
\usedim \log \left(1 + \frac{\Pcar}{\Tinit} \right) + 1 \Big \}
\log \left(\frac{\Tfinal}{\Tinit} \right).
\end{align}

Recall that the corollary statement applies to the KL divergence
$\KL(\Prob_\Zvar \| \Qprob)$, whereas the procedure thus far bounds
$\KL(\Prob_\Tinit \| \Qprob)$.  The natural intermediate object is
the divergence $\KL(\Prob_\Zvar \| \Prob_\Tinit)$, but
unfortunately, the KL divergence does not satisfy a triangle
inequality.  The following alternative result provides a way to
link these three objects:

\mygraybox{
  \begin{lemma}[KL transfer]
\label{LemKLTransfer}    
Consider probability distributions $(\Prob, \Qprob, \Sprob)$ such that
$\frac{d\Prob}{d\Sprob} \leq b \quad \Sprob\text{-a.s.}$ for some $b
\in [1, \infty)$. Then the square-root KL divergence satisfies the
  weak triangle inequality
\begin{align}
\label{EqnKLTransfer}  
  \sqrt{\KL(\Prob \| \Qprob)} & \leq \sqrt{\KL(\Prob \| \Sprob)} +
  \sqrt{b \KL(\Sprob \| \Qprob)}.
\end{align}
\end{lemma}
}
\noindent See~\Cref{SecProofLemKLTransfer} for the proof. \\

We apply~\Cref{LemKLTransfer} with the choices $\Prob = \Prob_\Zvar$,
$\Sprob = \Prob_\Tinit$ and $\Qprob$ corresponding to the sampler
output. With the specified iteration count~\eqref{EqnPoincareInter},
we have $\KL(\Prob_\Tinit \| \Qprob) \leq \varepsilon/2$, so that
the bound~\eqref{EqnKLTransfer}  implies that
\begin{align}
\label{EqnLooneyBean}  
  \sqrt{\KL(\Prob_Z \| \Qprob)} & \leq \sqrt{\KL(\Prob_Z \|
    \Prob_\Tinit)} + \sqrt{b \varepsilon/2} \; \leq \; \sqrt{\log b} +
  \sqrt{b \varepsilon/2} \qquad \mbox{where $b = \sup_{A}
    \Prob(A)/\Prob_\Tinit(A)$.}
\end{align}
Thus, it remains to bound $b = b(\Tinit)$ and choose $\Tinit$
appropriately.  By definition of $p_\Tinit$, we can write
\begin{align*}
p_\Tinit(x) = \Exs_W \big[p_Z(x - \sqrt{\Tinit} W) \big] \; = \;
\Exs_W \big[e^{-f(x - \sqrt{\Tinit} W)}] \qquad \mbox{where $f(z) \defn
  - \log p_Z(z)$.}
\end{align*}
Applying the assumed bound~\eqref{EqnLsmooth} with $u = -\sqrt{\Tinit}
W$ yields
\begin{align*}
  p_\Tinit(x) \geq \Exs_W \big[ e^{-f(x) + \sqrt{\Tinit}
      \inprod{\nabla f(x)}{W} - \tfrac{\Tinit L}{2} \|W \|_2^2} \big]
& \stackrel{(i)}{=} p_\Zvar(x) (1 + L \Tinit)^{-\usedim/2} \exp \Big
    \{ \frac{\Tinit \|\nabla f(x)\|_2^2}{2 (1 + L \Tinit)} \Big \} \\
& \stackrel{(ii)}{\geq}  p_\Zvar(x) (1 + L \Tinit)^{-\usedim/2},
\end{align*}
where step (i) follows by computing the expectation over the standard normal $W$; and
step (ii) follows since the exponential term is at least $1$.

Consequently, our required density ratio bound holds with $b(\Tinit) =
(1 + L \Tinit)^{\frac{\usedim}{2}}$, so that the bound~\eqref{EqnLooneyBean} implies
that
\begin{align}
  \label{EqnNearFinal}
  \sqrt{\KL(\Prob_Z \| \Qprob)} & \leq \sqrt{\tfrac{\usedim}{2}
  \log(1 + L \Tinit)} + \sqrt{(1 + L \Tinit)^{\frac{\usedim}{2}} \;
  \frac{\varepsilon}{2}}.
\end{align}
Setting $\Tinit = \frac{\varepsilon}{8 L \usedim}$ for any
$\varepsilon \in (0,1)$, the elementary bounds $\log(1 + x) \leq x$
and $(1 + x)^r \leq e^{r x}$ imply that
\begin{align*}
  \log b = \tfrac{\usedim}{2} \log(1 + L \Tinit) & \leq
  \frac{\varepsilon}{16},
  \qquad\mbox{and}\qquad
  b = (1 + L \Tinit)^{\frac{\usedim}{2}} \leq
  e^{\frac{\varepsilon}{16}}.
\end{align*}
Consequently, we have the upper bound
\begin{align*}
  \sqrt{\KL(\Prob_Z \| \Qprob)} & \leq \sqrt{\frac{\varepsilon}{16}} +
  e^{\frac{\varepsilon}{32}} \sqrt{\frac{\varepsilon}{2}} \leq \left\{
  \frac{1}{4} + \frac{e^{\frac{1}{32}}}{\sqrt{2}} \right\}
  \sqrt{\varepsilon} \leq \sqrt{\varepsilon},
\end{align*}
as claimed.

\subsubsection{\texorpdfstring{Proof of~\Cref{LemKLTransfer}}%
  {Proof of the KL transfer lemma}}
\label{SecProofLemKLTransfer}
By the KL chain rule, we have $\KL(\Prob \| \Qprob) = \KL(\Prob \|
\Sprob) + \KL(\Sprob \| \Qprob) + \Exs_\Sprob[(f-1) u]$, where $u =
\log(d \Sprob/d \Qprob)$ and $f = d \Prob/d \Sprob$.  Suppose that we
can show that
\begin{align}
\label{EqnFlatWhite}
  \Exs_\Sprob[(f-1) u] & \leq (b-1) \KL(\Sprob \| \Qprob) + 2 \sqrt{ b
    \, \KL (\Sprob \| \Qprob) \KL(\Prob \| \Sprob)}.
\end{align}
Then combining the preceding KL chain-rule identity
with~\eqref{EqnFlatWhite} yields
\begin{align*}
  \KL(\Prob \| \Qprob) & \leq \KL(\Prob \| \Sprob) + b \KL(\Sprob \|
\Qprob) + 2 \sqrt{ b \, \KL (\Sprob \| \Qprob) \KL(\Prob \| \Sprob)}
\; = \; \Big( \sqrt{\KL(\Prob \| \Sprob)} + \sqrt{b \KL(\Sprob \|
  \Qprob)} \Big)^2,
\end{align*}
as claimed.

\paragraph{Proof of the bound~\eqref{EqnFlatWhite}:}

Defining $\phi(t)=t + e^{-t} - 1$, we observe
\begin{align*}
  \Exs_\Sprob\big[ \phi\big( \log \tfrac{d \Sprob}{d \Qprob} \big) \Big] & = \KL(\Sprob
  \| \Qprob) + \Exs_\Sprob[ d \Qprob/d \Sprob] -1 \leq \KL(\Sprob \|
  \Qprob).
\end{align*}
Next we observe the bounds
\begin{align*}
-t \leq \sqrt{2\phi(t)} \quad\mbox{for $t \leq 0$, and} \quad t \leq
\phi(t) + \sqrt{2\phi(t)} \quad\mbox{for $t \geq 0$}.
\end{align*}
Recalling that $f = d \Prob/d \Sprob$, our assumption implies that the
function $f-1$ takes values in the interval $[-1, b-1]$.
Consequently, recalling the shorthand $u = \log (d \Sprob/ d \Qprob)$,
we have
\begin{align*}
(f-1) u \leq (b-1) \phi(u) + |f-1|\sqrt{2\phi(u)} 
\end{align*}
in a pointwise sense.  Taking expectations over $\Sprob$, we find that
\begin{align*}
  \Exp_\Sprob[(f-1)u] \leq (b-1) \KL(\Sprob \| \Qprob) +
  \sqrt{\underbrace{\Exs_\Sprob[(f-1)^2]}_{\equiv \chi^2(\Prob \|
      \Sprob)}} \; \sqrt{2 \underbrace{\Exs_\Sprob[\phi(u)]}_{\equiv
      \leq \KL(\Sprob \| \Qprob)}}
\end{align*}

To complete the proof, we need to show that $\chi^2(\Prob \| \Sprob)
\leq 2 b \KL(\Prob \| \Sprob)$.  To do so, we define the
function \mbox{$\psi(t) = t \log t - t +1$,} and observe that $f = d
\Prob/d \Sprob$ satisfies
\begin{align*}
  \Exs_\Sprob[\psi(f)] & = \Exs_\Sprob[ \tfrac{d \Prob}{d \Sprob} \log
    \tfrac{d \Prob}{d \Sprob}] - \Exs_\Sprob[\tfrac{d \Prob}{d
      \Sprob}] + 1 \; = \; \KL(\Prob \| \Sprob).
\end{align*}
Moreover, observe that $\psi(1) = \psi'(1) = 0$, and $\psi''(t) = 1/t
\geq 1/b$ for $t \in [0,b]$.  Since $f$ takes values in $[0,b]$, a
Taylor series expansion around $t = 1$ implies that $\frac{1}{2 b} (f
- 1)^2 \leq \psi(f)$ in a pointwise sense.  Taking expectations yields
$\frac{1}{2 b} \chi^2( \Prob \| \Sprob) \leq \KL (\Prob \| \Sprob)$ as
required.


\section{Proofs on accuracy of $\Btot$-block approximations}

In this section, we collect together the proofs of our results
on $\Btot$-block approximation.

\subsection{Proof of~\Cref{LemTwoPointSep}}
\label{SecProofLemTwoPointSep}

Without loss of generality (via a rescaling argument), it suffices to
prove the result for $\Tinit = 1$.  Write the original heat channel as
$X_t = R U + \sqrt{t} G$, where $U$ is Rademacher and $G \sim
\Normal(0, 1)$, and introduce the normalized binary channel $Y_s = U +
\sqrt{s} G$.  Define the MSE function
\begin{align*}
  m(s) & \defn \Exs\left[ \left( U - \Exs[U \mid Y_s] \right)^2
    \right], \quad \mbox{as well as} \quad \psi(u) \defn \frac{1}{2}
  m'(e^u).
\end{align*}
Letting $\hfun$ denote the MSE for the original channel, we then have
the relations $\hfun(t) = R^2 m(t/R^2)$, and $\hfun'(t) = m'(t/R^2)$.
Since $\qdens(r) = \hfun'(e^r)$ when $\Tinit = 1$, it follows that
$\psi(r - 2 \log R) = \frac{1}{2} \qdens(r)$.  Thus, using the
definitions of $\Hinfo$ and $\PartHfine$, we have
\begin{align}
\label{EqnDelToro}
  \Hinfo(a, b) & = \int_{\log(a/R^2)}^{\log(b/R^2)} \psi(u) \, du,
  \qquad \PartHfine(1, 2 R^2) = \left( \int_{-2 \log R}^{\log 2}
  \sqrt{\psi(u)} \, du \right)^2.
\end{align}

Now by the Gaussian-channel MMSE derivative
identity~\cite[Prop.~9]{GuoEtAl11}, applied in SNR
coordinates at SNR $1/s$, we have the relation $m'(s) = \frac{1}{s^2}
\Exs \big[ \Var\big( U \mid U + \sqrt{s} G \big)^2 \big]$.  Since $m(s) =\Exs \big[ \Var(U \mid Y_s) \big] \in [0,1]$, we have
\begin{align}
\label{EqnPilot}  
  0 \leq m'(s) \leq \frac{m(s)}{s^2}, \qquad 0 \leq \psi(u) \leq
  \frac{1}{2} m(e^u) e^{-2u}.
\end{align}
Moreover, the MSE function can be upper bounded using the estimator
$\widehat U \defn \operatorname{sign}(Y_s)$, whence
\begin{align*}
  m(s) & \stackrel{(i)}{\leq} \Exs\left[ \left( U - \widehat U \right)^2 \right] = 4
  \Phi(-1/\sqrt{s}) \; \stackrel{(ii)}{\leq} C \sqrt{s} e^{-1/(2s)}
\end{align*}
where step (i) follows since the conditional expectation is optimal;
and step (ii) follows from the Gaussian Mills-ratio
inequality~\cite{Gor41} $\Phi(-x) \leq \phi(x)/x$ for $x > 0$.
Combined with the first bound in equation~\eqref{EqnPilot}, we also
have $m'(s) \leq C s^{-3/2} e^{-1/(2s)}$ for $s \in (0, 1)$, and hence
\begin{align*}
  \sqrt{\psi(u)} & \leq C e^{-3u/4} \exp\left( -\frac{e^{-u}}{4}
  \right) \quad \mbox{for $u \leq 0$.}
\end{align*}
The substitution $z = e^{-u}$ shows that the super-exponential term
dominates the prefactor, and the right-hand side and its square are
integrable on $(-\infty, 0]$.

On $[0, \log 2]$, the derivative identity makes $\psi$ finite
and continuous.  It also implies $\psi(u) > 0$ for every finite $u$,
since the posterior variance of the non-degenerate Rademacher variable
is strictly positive at every finite noise level, so that
\begin{align*}
  0 < \int_{-\infty}^{\log 2} \sqrt{\psi(u)} \, du < \infty, \qquad 0
  < \int_{-\infty}^{\log 2} \psi(u) \, du < \infty.
\end{align*}
Using the integral representations~\eqref{EqnDelToro} and monotone
convergence, we conclude that
\begin{align*}
  \PartHfine(1, 2 R^2) & = \Theta(1), \quad \mbox{and} \quad \Hinfo(1,
  2 R^2) = \Theta(1).
\end{align*}
This completes the proof of the fine-partition claim. To prove the
single-block claim, we observe that
\begin{align*}
  \PartComp([1, 2 R^2]) & = \log(2 R^2) \; \Hinfo(1, 2 R^2) =
  \Theta(\log R).
\end{align*}

\paragraph{Two-block claim:}
Introduce the shorthand $L_R \defn \log(R^2)$ so that $\tau_R =
R^2/L_R$.  The two-block partition complexity is given by
$\PartComp(\Partition_2) = \big( \sqrt{\Sinfo_1 \Hinfo_1} +
\sqrt{\Sinfo_2 \Hinfo_2} \big)^2$, where
\begin{align*}
  \Sinfo_1 & = \log \tau_R = 2 \log R - \log L_R, \qquad \Sinfo_2 =
  \log(2 R^2/\tau_R) = \log(2 L_R) = \Theta(\log\log R).
\end{align*}
and $\Hinfo_1 \defn \Hinfo(1, \tau_R)$ and $\Hinfo_2 \defn
\Hinfo(\tau_R, 2 R^2)$.  It suffices to show that $\Sinfo_1 \Hinfo_1 =
o(1)$ and $\Sinfo_2 \Hinfo_2 = \Theta(\log \log R)$.

In terms of $\psi$, the two increments can be expressed as
\begin{align*}
  \Hinfo_1 & = \int_{-2 \log R}^{-\log L_R} \psi(u) \, du, \qquad
  \Hinfo_2 = \int_{-\log L_R}^{\log 2} \psi(u) \, du.
\end{align*}
For the first block, the small-noise estimate and the substitutions $s
= e^u$ and $z = 1/(2s)$ give
\begin{align*}
  \Hinfo_1 & \leq c \int_{R^{-2}}^{1/\log(R^2)} s^{-5/2} e^{-1/(2s)}
  \, ds \leq c \int_{L_R/2}^{R^2/2} z^{1/2} e^{-z} \, dz \leq c
  \frac{\sqrt{\log R}}{R}.
\end{align*}
Here the last inequality uses $\int_a^\infty z^{1/2} e^{-z} \, dz \leq
c a^{1/2} e^{-a}$ for $a$ sufficiently large.  Since $\Sinfo_1 =
\Theta(\log R)$, this gives $\Sinfo_1 \Hinfo_1 = o(1)$.  On the other
hand, monotone convergence, as well as the positivity and
integrability proved above, guarantee that $\Hinfo_2 \longrightarrow
\int_{-\infty}^{\log 2} \psi(u) \, du \in (0, \infty)$.  Putting
together the pieces yields $\Sinfo_2 \Hinfo_2 = \Theta(\log\log R)$,
which completes the proof.


\subsection{Proof of~\Cref{LemBlockAccuracy}}
\label{SecProofLemBlockAccuracy}

For a log-time partition $\Partition = \{ I_\bind \}_{\bind=0}^{\Btot
  - 1}$, where $I_\bind = [r_\bind, r_{\bind + 1}]$ and $\Sinfo_\bind
= r_{\bind + 1} - r_\bind$, we define mean and variance parameters as
\begin{subequations}
\begin{align}
\label{EqnBlockAccuracyMoments}
  \mu_\bind & \defn \frac{1}{\Sinfo_\bind} \int_{I_\bind} \fdens(r) \,
  dr, \qquad v_\bind \defn \frac{1}{\Sinfo_\bind} \int_{I_\bind} \big(
  \fdens(r) - \mu_\bind \big)^2 \, dr.
\end{align}
With this representation, we have the relation $\Hblk_\bind =
\frac{\Sinfo_\bind}{2}(\mu_\bind^2 + v_\bind)$, and hence, from the
definitions of $\PartComp(\Partition)$ and $\PartHfine$, we have an
expression
\begin{align}
\label{EqnBlockAccuracyExactGap}
  \sqrt{2 \PartComp(\Partition)} - \sqrt{2 \PartHfine(\Tinit,
    \Tfinal)} & = \sum_{\bind=0}^{\Btot - 1} \Sinfo_\bind \left(
  \sqrt{\mu_\bind^2 + v_\bind} - \mu_\bind \right) \; = \;
  \sum_{\bind=0}^{\Btot - 1} \frac{\Sinfo_\bind v_\bind}
      {\sqrt{\mu_\bind^2 + v_\bind} + \mu_\bind}.
\end{align}
\end{subequations}
for the error in the approximation at the square-root level.

For each block, the elementary range bound for a variance combined
with the definition of the modulus yields the upper bound $v_\bind \leq
\frac{1}{4} \operatorname{osc}_{I_\bind}^2(\fdens) \leq \frac{1}{4}
\omega_\fdens^2(\Sinfo_\bind)$, whereas $\sqrt{\mu_\bind^2 + v_\bind}
- \mu_\bind \leq \sqrt{v_\bind}$.  Summing these bounds in
equation~\eqref{EqnBlockAccuracyExactGap}, and using
$\sum_{\bind=0}^{\Btot - 1} \Sinfo_\bind = \Len$, gives the
bound~\eqref{EqnBlockAccuracyEqual}.

When $\fdens \geq m$, we have $\mu_\bind \geq m$, so the denominator
in equation~\eqref{EqnBlockAccuracyExactGap} is at least $2 m$.  The
variance identity $v_\bind = \frac{1}{2 \Sinfo_\bind^2} \int_{I_\bind}
\int_{I_\bind} \big( \fdens(r) - \fdens(s) \big)^2 \, dr \, ds$
implies, for every uniform block of length $h = \Len / \Btot$, that
\begin{align*}
  v_\bind \leq \frac{1}{h^2} \int_0^h (h - u) \omega_\fdens(u)^2 \,
  du.
\end{align*}
Substitution into equation~\eqref{EqnBlockAccuracyExactGap}, followed
by summation over the uniform blocks yields the
bound~\eqref{EqnBlockAccuracyIntegrated}.

Finally, the Lipschitz bound $\omega_\fdens(u) \leq \Lip(\fdens) u$
gives $\frac{\Len}{2} \omega_\fdens(h) \leq \frac{\Lip(\fdens) \Len
  h}{2} = \frac{\Lip(\fdens) \Len^2}{2 \Btot}$, which proves the
bound~\eqref{EqnBlockAccuracyLipschitz}(a), while
\begin{align*}
  \frac{\Len}{2 m h^2} \int_0^h (h - u) \omega_\fdens(u)^2 \, du &
  \leq \frac{\Lip(\fdens)^2 \Len}{2 m h^2} \int_0^h (h - u) u^2 \, du
  \; = \; \frac{\Lip(\fdens)^2 \Len h^2}{24 m} = \frac{\Lip(\fdens)^2
    \Len^3}{24 m \Btot^2}
\end{align*}
thereby establishing~\eqref{EqnBlockAccuracyLipschitz}(b).


\section{Guarantees for approximate structure}
\label{SecApproxStructure}
In this section, we describe various extensions of the examples
following~\Cref{PropGoodRateDistor}, in particular allowing for
approximate structure of various types.

\subsection{Consequences for approximate manifold structure}
Assuming that $\Zvar$ lies exactly on a manifold is a rather brittle
condition; it is more realistic to assume that $\Zvar$ can be
approximated by a variable on a manifold.  Concretely, letting
$\Wass_2$ denote the Wasserstein-$2$ distance, suppose that $\Zvar$
lies at $\Wass_2$-distance at most $\eta$ from a compact
$\mdim$-dimensional manifold $\Mani$.  By definition of the
$\Wass_2$-distance, there exists a coupling $(\Zvar, \Ztil)$ with a
random variable $\Ztil$ supported on $\Mani$ such that $\Exs \|\Zvar -
\Ztil\|_2^2 \leq \eta^2$.  Now let $\Zhat$ take values over a
$t$-accurate covering of $\Mani$ with $M$ elements, so that $\log M
\leq c_0 + \mdim \log(1 + c_1/t)$ by the manifold structure.  As in
our earlier argument, we have $\ShanInfo(\Zvar; \Zhat) \leq \log M$,
and moreover, we have $\Exs \|\Zvar - \Zhat\|_2^2 \leq 2 \Exs \|\Zvar
- \Ztil\|_2^2 + 2 \Exs \|\Ztil - \Zhat\|_2^2 \leq 2 \big( \eta^2 + t^2
\big)$.  We then apply the inequality~\eqref{EqnGoodRateDistorWeak}
with these bounds, thereby finding that
\begin{align}
\label{EqnApproxManifold}
\Hinfo(\Tinit, \Tfinal) & \leq c_0 + \inf_{t > 0} \Big \{\mdim \log
\big(1 + (c_1/t) \big) + \frac{t^2}{\delta} \Big \} +
\frac{\eta^2}{\delta}.
\end{align}
When $\eta = 0$, we recover the exact manifold case, whereas in
general, we pay an additional price that scales as $\eta^2/\delta$.

\subsection{Consequences for approximate Gaussian mixture models}
\label{SecApproxStructureMix}
We begin by observing a natural approximation-theoretic guarantee provided by
our theory.  Consider a discrete distribution $\pi$ supported on $K$
points with entropy $\Ent(\pi)$.  Applying the
bound~\eqref{EqnGoodRateDistorWeak} guarantees that
\begin{subequations}
\begin{align}
  \label{EqnInitalBike}
  \Hinfo(\delta, T) & \leq \Ent(\pi) + \frac{1}{2 \delta}
  \Wass_2^2(\Prob_\Zvar, \nu) + \frac{\hfun(T)}{T}.
\end{align}
Consequently, if $\Prob_\Zvar$ is approximated to $\Wass_2$-accuracy
$\eta$ by a $K$-point distribution, then we have
\begin{align}
  \label{EqnBike}
  \Hinfo(\delta, T) & \leq \Ent(\pi) + \frac{\eta^2}{2 \delta} +
  \frac{\hfun(T)}{T}.
\end{align}
\end{subequations}
Now consider sampling from the distribution $\Prob_\delta$, which is a
convolved version of $\Prob_Z$.  Our theory then guarantees that we
can sample $\Xhat_N \sim \Qprob_\delta$ such that $\KL(\Prob_\delta \|
\Qprob_\delta) \leq \varepsilon$ using $N \asymp
\Hinfo(\delta, T)/\varepsilon$ iterations, where $\Hinfo$ satisfies the
upper bound~\eqref{EqnBike}.  Thus, we have obtained an
approximation-theoretic generalization of our
guarantee~\eqref{EqnGaussMixNscore} for sampling from exact Gaussian
mixture distributions, again with the excess term scaling as
$\eta^2/\Tinit$.


\end{document}